\documentclass[10pt,a4paper]{amsart}
\usepackage[utf8]{inputenc}
\usepackage{amsmath,amssymb,amsthm,enumerate,ulem,color}
\usepackage{hyperref} 
\usepackage[shortlabels]{enumitem}
\usepackage{hyperref}
\usepackage{xcolor}
\newtheorem{theorem}{Theorem}[section]
\newtheorem{thmx}{Theorem}

\newtheorem{proposition}[theorem]{Proposition}
\newtheorem{corollary}[theorem]{Corollary}
\newtheorem{observation}[theorem]{Observation}
\newtheorem{lemma}[theorem]{Lemma}

\theoremstyle{definition}
\newtheorem{definition}[theorem]{Definition}
\newtheorem{notation}[theorem]{Notation}
\newtheorem*{thm*}{Theorem}
\newcounter{claimcounter}[section]
\newtheorem{claim}[claimcounter]{Claim}
\newtheorem*{claim*}{Claim}

\newtheorem{remark}[theorem]{Remark}
\newtheorem{example}[theorem]{Example}

\newcounter{que}
\newtheorem{question}[que]{Question}

\newcommand{\N}{\mathbb N}
\newcommand{\setsep}{:\;}

\newcommand\isomtrclass[2][]{\langle #2\rangle_{\equiv}^{#1}}
\newcommand\isomrfclass[2][]{\langle #2\rangle_{\simeq}^{#1}}
\newcommand\approximates[2][]{\stackrel{#2}{\sim} #1}
\newcommand{\Span}{\operatorname{span}}

\newcommand\closedSpan[1]{\overline{\Span} \{#1\}}
\def\Gurarii{\mathbb{G}}

\def\F{\mathcal{F}}

\def\B{\mathcal{B}}

\def\PP{\mathcal{P}}

\def\dist{\operatorname{dist}}

\newcommand{\Rat}{\mathbb{Q}}
\newcommand{\Rea}{\mathbb{R}}
\newcommand{\Nat}{\mathbb{N}}
\newcommand{\dom}{\operatorname{dom}}

\newcommand{\qeSpan}{\Span_\Rat}
\def\b{\mathfrak{b}}

\begin{document}
\normalem 

\title[Polish spaces of Banach spaces]{Polish spaces of Banach spaces}

\author[M. C\' uth]{Marek C\'uth}
\author[M. Dole\v{z}al]{Martin Dole\v{z}al}
\author[M. Doucha]{Michal Doucha}
\author[O. Kurka]{Ond\v{r}ej Kurka}
\email{cuth@karlin.mff.cuni.cz}
\email{dolezal@math.cas.cz}
\email{doucha@math.cas.cz}
\email{kurka.ondrej@seznam.cz}

\address[M.~C\' uth]{Charles University, Faculty of Mathematics and Physics, Department of Mathematical Analysis, Sokolovsk\'a 83, 186 75 Prague 8, Czech Republic}
\address[M.~Dole\v{z}al, M.~Doucha, O.~Kurka]{Institute of Mathematics of the Czech Academy of Sciences, \v{Z}itn\'a 25, 115 67 Prague 1, Czech Republic}


\thanks{M. C\'uth was supported by Charles University Research program No. UNCE/SCI/023 and by the Czech Science Foundation, project no. GA\v{C}R 17-00941S. Research of M. Dole\v zal was supported by the GA\v CR project 17-27844S and RVO: 67985840. M. Doucha was supported by the GA\v{C}R project 19-05271Y and RVO: 67985840.
 O. Kurka was supported by the Czech Science Foundation, project no. GA\v{C}R 17-00941S and RVO: 67985840.
}

\begin{abstract}
   We present and thoroughly study natural Polish spaces of separable Banach spaces. These spaces are defined as spaces of norms, resp. pseudonorms, on the countable infinite-dimensional rational vector space. We provide an exhaustive comparison of these spaces with admissible topologies recently introduced by Godefroy and Saint-Raymond and show that Borel complexities differ little with respect to these two different topological approaches.
   
   We investigate generic properties in these spaces and compare them with those in admissible topologies, confirming the suspicion of Godefroy and Saint-Raymond that they depend on the choice of the admissible topology.
\end{abstract}
\maketitle

\section*{Introduction}
Banach spaces and descriptive set theory have a long history of mutual interactions. An explicit use of descriptive set theory to Banach space theory can be traced back at least to the seminal papers of Bourgain (\cite{Bourgain1}, \cite{Bourgain2}), where it has become apparent that descriptive set theory is an indispensable tool for universality problems. That is a theme that has been investigated by researchers working with Banach spaces ever since (see e.g. \cite{ArgDod} and \cite{dodosKniha} and references therein).

As it eventually turned out, `Descriptive set theory of Banach spaces' is an interesting and rich subject of its own and it has received some considerable attention in the recent years. One of the starting points was the idea of Bossard of coding separable Banach spaces in \cite{Bossard93}, \cite{Bo02}. His approach was, which can be considered standard, to choose some universal separable Banach space $X$, e.g. $C(2^\Nat)$, and consider the Effros-Borel space $F(X)$. Recall that this is the set of closed subsets of $X$ equipped with a certain $\sigma$-algebra which makes $F(X)$ a standard Borel space, i.e. a measurable space which is isomorphic, as a measurable space, to a Polish space equipped with the $\sigma$-algebra of Borel sets. It is then not too difficult to show that the subset $SB(X)\subseteq F(X)$ consisting of all closed linear subspaces is a Borel subset, and therefore a standard Borel space itself.

Although this approach has found numerous significant applications in Banach space theory, its drawback is that there is no canonical or natural (Polish) topology on $SB(X)$. So although one can ask whether a given class of Banach spaces is Borel or not, the question about the exact complexity of that particular class is meaningless. Let us specify this. Many of the applications may be interpreted as computing co/analytic sets of Banach spaces and deriving consequences from this; this concerns especially various universality results, see e.g. \cite[Chapter 7]{dodosKniha} and \cite{Ku16}. Having a topology allows us to separate two classes of Banach spaces which are both known to be Borel (we comment more on this issue in the sequel \cite{CDDK21}).

Moreover, one of the  active and ongoing research streams is to find out whether for a particular Banach space its isomorphism class is Borel or not (see e.g. \cite{K19}, \cite{Godef17}, or the survey \cite{G17} and references therein): spaces with Borel isomorphism classes are rare and can be considered simply definable (up to isomorphism). It is then desirable, for spaces whose classes are Borel, to have a finer description how simply definable they are, see e.g. \cite[Problem 3]{Godef17}.

A recent work \cite{GS18} of Godefroy and Saint-Raymond addresses this general issue of associating a natural topology to the set of codes of Banach spaces. They still work with the space $SB(X)$, but among the many Polish topologies on $SB(X)$ giving the Effros-Borel structure, they select some particular subclass which is called \emph{admissible topologies}. Although no particular admissible topology is canonical, the set of requirements put on this class guarantees that the exact Borel complexities vary little.

This paper presents an alternative approach by considering a concrete and natural Polish space (and some variants of it) of separable Banach spaces, which is convenient to work with. 
We have three main reasons and advantages for that in mind. The first one, which was our original motivation, is to push much further the programme initiated by Godefroy and Saint-Raymond on computing precise Borel complexities of various classes of Banach spaces. It turns out that in our new space the computations of Borel complexities are usually as straightforward as they could be, and besides many new results we are also able to improve several estimates already obtained in \cite{GS18} (see also \cite{Gh19} for additional results in this direction). Most of these results are contained in the sequel to this paper \cite{CDDK21}. It should be also mentioned that computing exact Borel complexities has been a traditional research topic in analysis and topology of independent interest (see \cite[Chapter 23]{KechrisBook} for a comprehensive but already outdated list of examples)
and one of our contributions in that regard, presented in \cite{CDDK21}, are new elegant characterizations of the Hilbert space $\ell_2$. Shortly, $\ell_2$ is the unique (up to isometry) infinite-dimensional separable Banach space with a  closed isometry class, and it is the unique (up to isomorphism) infinite-dimensional separable Banach space with an $F_{\sigma}$ isomorphism class. Recall that Bossard \cite[Problem 2.9]{Bo02} originally asked whether $\ell_2$ is the unique space with a Borel isomorphism class. Although this is known to be false now, we can show that no other Banach space can have such a simple isomorphism class.

The remaining two reasons have two different origins, one is functional analytic, the other is coming from logic. For the former, the feature of the topology we work with is that basic open sets are essentially definitions of finite-dimensional Banach spaces up to $\varepsilon$-isomorpism, where $\varepsilon>0$ is arbitrarily small. This connects this topological approach with the local theory of Banach spaces. The basic manifestation of this is that finite representability of Banach spaces is expressed in elementary topological terms and for example leads to a natural reformulation of the Dvoretzky theorem that the infinite-dimensional separable Hilbert space is contained in the closure of the isometry class of every infinite-dimensional separable Banach space (see Corollary~\ref{cor:Dvoretzky}). Further applications to the local theory might be a subject of future research.

Lastly, our approach brings closer the two different interactions of logic with Banach space theory that had not interacted significantly. That is, the descriptive set theory of Banach spaces and the continuous model theory of Banach spaces. Continuous model theory, or model theory of metric structures, is a generalization of the classical model theory to structures that are inherently metric, and has its origins, motivations as well as most of the applications in Banach space theory. We refer the reader to \cite{BYBHU} for an introduction and more motivation. Our space of Banach spaces is closely related to how countable, resp. separable models are coded in classical, resp. continuous model theory (see e.g. \cite[Section 3.6]{Gao} for the classical case and \cite[Section 4]{BYDNT} for the metric case). Moreover, the exact Borel complexities that our space allows to compute are directly related to L\'opez-Escobar theorem from continuous infinitary logic which connects such complexities with the complexities of formulas that define the corresponding classes (we refer to \cite[Section 6]{BYDNT}). It may be also of interest for a future research to investigate the relation, for a given Banach space, between having an isometry class of low Borel complexity and having an axiomatization in continuous first-order logic.

Having motivated our approach, let us now outline some more details and the main results contained in this paper. Informally, the space we introduce is the space of all norms, resp. pseudonorms, on the space of all finitely supported sequences of rational numbers -- the unique infinite-dimensional vector space over $\Rat$ with a countable Hamel basis. This is also, in spirit, similar to how (for instance) Vershik topologized the space of all Polish metric spaces (\cite{Ver}), or how Grigorchuk topologized the space of all $n$-generated, resp. finitely generated, groups (\cite{Gri}).

This space has already appeared in previous works of the authors in \cite{CDK21a} and \cite{CDK21b} as a useful coding of Banach spaces. Here we investigate it further.

Some of the main results of this paper are listed now. The first theorem below presents the main part of the comparison of the space of norms with admissible topologies, whose proof is the core of Section~\ref{sect:choice}, where also all other comparison results are proved.

\begin{thmx}\label{thm:Intro1}
There is a $\boldsymbol{\Sigma}_2^0$-measurable mapping from the space of norms to any admissible topology of Godefroy and Saint-Raymond that associates to each norm a space isometric to the space which the norm defines, and vice versa.
\end{thmx}

Therefore, while the exact Borel complexities are more or less independent of our coding or the choice of the admissible topology, some finer topological properties, such as being meager or comeager (this is mentioned below), or the description of the topological closures, are not. We obtain a neat characterization of topological closures in the spaces of norms and pseudonorms in terms of finite representability, we refer the reader to Proposition~\ref{prop:closureOfIsometryClasses}.

Then, directly motivated by \cite[Problem 5.5]{GS18}, we investigate the generic properties in the spaces of norms and in admissible topologies. First we show the genericity of the Gurari\u{\i} space in the space of norms and pseudonorms.

\begin{thmx}\label{thm:Intro4}
The isometry class of the Gurari\u{\i} space is a dense $G_\delta$-set in the space of norms and pseudonorms, i.e. the Gurari\u{\i} space is the generic separable Banach space (see Theorem~\ref{thm:gurariiTypicalInP}).
\end{thmx}
Then we continue with a similar investigation for admissible topologies and among other things, we show that the Gurari\u{\i} space is not always generic.
\begin{thmx}\label{thm:Intro5}
For any isometrically universal separable Banach space $X$, any admissible topology $\tau$ on $SB(X)$ and any infinite-dimensional Banach spaces $Y$ and $Z$ such that $Y\hookrightarrow Z$ and $Z\not\hookrightarrow Y\oplus F$ for every finite-dimensional space $F$, there exists a finer admissible topology $\tau'\supseteq \tau$ such that the class of Banach spaces isomorphic to $Z$ is nowhere dense in $(SB(X),\tau')$. In particular, there exists an admissible topology in which the Gurari\u{\i} space is meager (see Theorem~\ref{thm:Baireadmiss2}).
\medskip

On the other hand, the isometry class of the Gurari\u{\i} space $\mathbb{G}$, as a subset of $SB(\mathbb{G})$, is a dense $G_\delta$-set in the Wijsman topology (see Theorem~\ref{thm:GurariiinWijsman}).
\end{thmx}

\subsection{Notation}

Let us conclude the introduction by setting up some notation that will be used throughout the paper.\\

Throughout the paper we usually denote the Borel classes of low complexity by the traditional notation such as $F_\sigma$ and $G_\delta$, or even $F_{\sigma \delta}$ (countable intersection of $F_\sigma$ sets) and $G_{\delta \sigma}$ (countable union of $G_\delta$-sets). However, whenever it is more convenient or necessary we use the notation $\boldsymbol{\Sigma}_\alpha^0$, resp. $\boldsymbol{\Pi}_\alpha^0$, where $\alpha<\omega_1$ (we refer to \cite[Section 11]{KechrisBook} for this notation). We emphasize that open sets, resp. closed sets, are $\boldsymbol{\Sigma}^0_1$, resp. $\boldsymbol{\Pi}^0_1$, by this notation.

Moreover, given a class $\boldsymbol{\Gamma}$ of sets in metrizable spaces, we say that $f:X\to Y$ is $\boldsymbol{\Gamma}$-measurable if $f^{-1}(U)\in \boldsymbol{\Gamma}$ for every open set $U\subseteq Y$.

Given Banach spaces $X$ and $Y$, we denote by $X\equiv Y$ (resp. $X\simeq Y$) the fact that those two spaces are linearly isometric (resp. isomorphic). We denote by $X\hookrightarrow Y$ the fact that $Y$ contains a subspace isomorphic to $X$. For $K\geq 1$, a \emph{$K$-isomorphism} $T:X\to Y$ is a linear map with $K^{-1}\|x\|\leq\|Tx\|\leq K\|x\|$, $x\in X$. If $x_1,\ldots,x_n$ are linearly independent elements of $X$ and $y_1,\ldots,y_n\in Y$, we write $(Y,y_1,\ldots,y_n)\approximates[(X,x_1,\ldots,x_n)]{K}$ if the linear operator $T:\Span\{x_1,\ldots,x_n\}\to\Span\{y_1,\ldots,y_n\}$ sending $x_i$ to $y_i$ satisfies $\max\{\|T\|,\|T^{-1}\|\}<K$. If $X$ has a canonical basis $(x_1,\ldots,x_n)$ which is clear from the context, we just write $(Y,y_1,\ldots,y_n)\approximates[X]{K}$ instead of $(Y,y_1,\ldots,y_n)\approximates[(X,x_1,\ldots,x_n)]{K}$. Morevoer, if $Y$ is clear from the context, we write $(y_1,\ldots,y_n)\approximates[X]{K}$ instead of $(Y,y_1,\ldots,y_n)\approximates[X]{K}$.

Throughout the text $\ell_p^n$ denotes the $n$-dimensional $\ell_p$-space, i.e. the upper index denotes dimension. Finally, in order to avoid any confusion, we emphasize that if we write that a mapping is an ``isometry'' or an ``isomorphism'', we do not mean it is surjective if this is not explicitly mentioned.

\section{The Polish spaces \texorpdfstring{$\PP_\infty$}{Pinfty} and \texorpdfstring{$\B$}{B}, their basic properties}\label{sect:newSpaces}

In this section we introduce the main notions of this paper: the Polish spaces of pseudonorms $\PP$ (and $\PP_\infty$) representing separable (infinite-dimensional) Banach spaces, and we recall the space of norms $\B$ that has appeared in our previous works \cite{CDK21a} and \cite{CDK21b}. We show some interesting features of these spaces, e.g. the neat relation between finite representability and topological closures in these spaces; see Proposition~\ref{prop:closureOfIsometryClasses} and its corollaries.

Let us start with the following idea of coding the class of separable (infinite-dimensional) Banach spaces. It is based on the idea presented already in our previous papers \cite{CDK21a} and \cite{CDK21b}, where the space $\B$ was defined.
\bigskip

By $V$, we shall denote the vector space over $\Rat$ of all finitely supported sequences of rational numbers; that is, the unique infinite-dimensional vector space over $\Rat$ with a countable Hamel basis $(e_n)_{n\in\Nat}$.

\begin{definition}\label{def:basic}
Let us denote by $\PP$ the space of all pseudonorms on the vector space $V$. Since $\PP$ is a closed subset of $\Rea^V$, this gives $\PP$ the Polish topology inherited from $\Rea^V$. The subbasis of this topology is given by sets of the form $U[v,I]:=\{\mu\in\PP\setsep \mu(v)\in I\}$, where $v\in V$ and $I$ is an open interval.

We often identify $\mu\in\PP$ with its extension to the pseudonorm on the space $c_{00}$, that is, the vector space over $\Rea$ of all finitely supported sequences of real numbers.

For every $\mu\in\PP$ we denote by $X_\mu$ the Banach space given as the completion of the quotient space $X/N$, where $X = (c_{00},\mu)$ and $N = \{x\in c_{00}\setsep \mu(x) = 0\}$. In what follows we often consider $V$ as a subspace of $X_\mu$, that is, we identify every $v\in V$ with its equivalence class $[v]_N\in X_\mu$.

By $\PP_\infty$ we denote the set of those $\mu\in\PP$ for which $X_\mu$ is infinite-dimensional Banach space. As we did in \cite{CDK21a} and \cite{CDK21b}, by $\B$ we denote the set of those $\mu\in\PP_\infty$ for which the extension of $\mu$ to $c_{00}$ is an actual norm, that is, the vectors $ e_{1}, e_{2}, \dots $ are linearly independent in $X_\mu$.

We endow $\PP_\infty$ and $\B$ with the topologies inherited from $\PP$.
\end{definition}

Our first aim is to show that the topologies on $\PP_\infty$ and $\B$ are Polish (see Corollary~\ref{cor:polishTopologies}). This can be easily verified directly, here we obtain it as a corollary of the fact that the relation $\approximates{K}$ defined before is open in $\PP$, a very useful fact that will prove important many times in the paper.

We need the following background first. Given a metric space $(M,d)$, $\varepsilon>0$ and $N,S\subseteq M$ we say that $N$ is \emph{$\varepsilon$-dense for $S$} if for every $x\in S$ there is $y\in N$ with $d(x,y)<\varepsilon$ (let us emphasize that we do not assume $N\subseteq S$). For further references, we recall the following well-known approximation lemma, for a proof see e.g. \cite[Lemma 12.1.11]{albiacKniha}.
\begin{lemma}\label{lem:approx}
There is a function $\phi_1:[0,1)\to[0,1)$ continuous at zero with $\phi_1(0)=0$ such that whenever $T:E\to X$ is a linear operator between Banach spaces, $\varepsilon\in(0,1)$, $M\subseteq E$ is $\varepsilon$-dense for $S_E$ and
    \[
    \forall m\in M: |\|Tm\| - 1|< \varepsilon,
    \]
    then $T$ is a $(1+\phi_1(\varepsilon))$-isomorphism between $E$ and $T(E)$.
\end{lemma}
The following definition precises the notation $\approximates{K}$ defined in the introduction.
\begin{definition}
If $v_1,\ldots,v_n\in V$ are given, for $\mu\in\PP$, instead of $(X_\mu,v_1,\ldots,v_n)\approximates[X]{K}$, we shall write $(\mu,v_1,\ldots,v_n)\approximates[X]{K}$.
\end{definition}
\begin{lemma}\label{lem:infiniteDimIsGDelta}
Let $X$ be a Banach space with $\{x_1,\ldots,x_n\}\subseteq X$ linearly independent and let $v_1,\ldots,v_n\in V$. Then for any $K>1$ the set $$\mathcal{N}((x_i)_i,K,(v_i)_i)=\{\mu\in \PP\colon (\mu,v_1,\ldots,v_n)\approximates[(X,x_1,\ldots,x_n)]{K} \}$$ is open in $\PP$.

In particular, the set of those $\mu\in\PP$ for which the set $\{v_1,\ldots,v_n\}$ is linearly independent in $X_\mu$ is open in $\PP$.
\end{lemma}
\begin{proof}Pick some $\mu\in \mathcal{N}((x_i)_i,K,(v_i)_i)$. By definition, the linear map $T$ sending $v_i$ to $x_i\in X$, $i\leq n$, is a linear isomorphism satisfying $\max\{\|T\|,\|T^{-1}\|\}<L$ for some $L<K$. Let $\phi_1$ be the function provided by Lemma~\ref{lem:approx} and pick $\varepsilon>0$ such that $L(1 + \phi_1(2\varepsilon))<K$.
Let $N\subseteq V$ be a finite $\varepsilon$-dense set for the sphere of $\Span\{v_1,\ldots,v_n\}\subseteq X_\mu$ such that $\mu(v)\in(1-\varepsilon,1+\varepsilon)$ for every $v\in N$. Then
\[
U:=\{\nu\in\PP\setsep \forall v\in N:\; |\nu(v)-\mu(v)|<\varepsilon\}
\]
is an open neighborhood of $\mu$ and $U\subseteq \mathcal{N}((x_i)_i,K,(v_i)_i)$. Indeed, for any $\nu\in U$ we have that $id:(\Span\{v_1,\ldots,v_n\},\mu)\to (\Span\{v_1,\ldots,v_n\},\nu)$ is a $(1+\phi_1(2\varepsilon))$-isomorphism; hence, the linear map $T$ considered as a map betweeen $(\Span\{v_1,\ldots,v_n\},\nu)$ and $\Span\{x_1,\ldots,x_n\}$ satisfies $\|T\|<L(1 + \phi_1(2\varepsilon))<K$, and similarly $\|T^{-1}\|<K$; hence, $\nu\in \mathcal{N}((x_i)_i,K,(v_i)_i)$.

The ``In particular'' part easily follows, because $v_1,\ldots,v_n\in V$ are linearly independent if and only if there exists $K>1$ with $(\mu,v_1,\ldots,v_n)\approximates[\ell_1^n]{K}$.
\end{proof}

\begin{corollary}\label{cor:polishTopologies}
Both $\PP_\infty$ and $\B$ are $G_\delta$-sets in $\PP$.
\end{corollary}

Since we are interested mainly in subsets of $\PP$ closed under isometries, we introduce the following notation.

\begin{notation}\label{not:ismfrClass}
Let $Z$ be a separable Banach space and let $\mathcal{I}$ be a subset of $\PP$. We put
\[\isomtrclass[\mathcal{I}]{Z}:=\{\mu\in\mathcal{I}\setsep X_\mu\equiv Z\} \quad \text{and} \quad \isomrfclass[\mathcal{I}]{Z}:=\{\mu\in\mathcal{I}\setsep X_\mu\simeq Z\}.\]
If $\mathcal{I}$ is clear from the context we write $\isomtrclass{Z}$ and $\isomrfclass{Z}$ instead of $\isomtrclass[\mathcal{I}]{Z}$ and $\isomrfclass[\mathcal{I}]{Z}$, respectively.
\end{notation}

The important feature of the topology of the spaces $\PP$, $\PP_\infty$ and $\B$ is that basic open neighborhoods are defined using finite data, i.e. finitely many vectors. That suggests that topological properties of the aforementioned spaces should be closely related to the local theory of Banach spaces. This is certainly a point that could be investigated further in a future research. Here we just observe how topological closures are related to finite representability, see Proposition~\ref{prop:closureOfIsometryClasses}. In order to formulate our results, let us consider the following generalization of the classical notion of finite representability.

\begin{definition}
Let $\F$ be a family of Banach spaces. We say that a Banach space $X$ is finitely representable in $\F$ if given any finite-dimensional subspace $E$ of $X$ and any $\varepsilon > 0$ there exists a finite-dimensional subspace $F$ of some $Y\in\F$ which is $(1+\varepsilon)$-isomorphic to $E$.

If the family $\F$ consists of one Banach space $Y$ only, we say rather that $X$ is finitely representable in $Y$ than in $\{Y\}$.

If $\F\subseteq \PP$, by saying that $X$ is finitely representable in $\F$ we mean it is finitely representable in $\{X_\mu\setsep \mu\in\F\}$.
\end{definition}

The following is an easy observation which we will use further, the proof follows e.g. immediately from \cite[Lemma 12.1.7]{albiacKniha} in the case that $\F$ contains one Banach space only. For the more general situation the proof is analogous.
\begin{lemma}\label{lem:finiteReprLemma}
Let $\F$ be a family of infinite-dimensional Banach spaces and $\mu\in\PP_\infty$. Let $\{k(n)\}_{n=1}^\infty$ be a sequence such that $\{e_{k(n)}\setsep n\in\Nat\}$ is a linearly independent set in $X_\mu$ and $\overline{\Span}\{e_{k(n)}\setsep n\in\Nat\}=X_\mu$. Then $X_\mu$ is finitely representable in $\F$ if and only if for every $n\in\Nat$ and $\varepsilon>0$ there exists a finite-dimensional subspace $F$ of some $Y\in\F$ which is $(1+\varepsilon)$-isomorphic to $(\Span\{e_{k(1)},\ldots,e_{k(n)}\},\mu)$.
\end{lemma}
\begin{proposition}\label{prop:closureOfIsometryClasses}
Let $\F\subseteq \B$ be such that $\isomtrclass[\B]{X_\mu}\subseteq \F$ for every $\mu\in\F$. Then
\[\{\nu\in\B\setsep X_\nu\text{ is finitely representable in }\F\}= \overline{\F}\cap \B.\]
    The same holds if we replace $\B$ with $\PP_\infty$ or with $\PP$.
    
    In particular, if $X$ is a separable infinite-dimensional Banach space, then
    \[\{\nu\in\B\setsep X_\nu\text{ is finitely representable in }X\} =  \overline{\isomtrclass[\B]{X}}\cap\B,\]
    and similarly also if we replace $\B$ with $\PP_\infty$ or with $\PP$.
\end{proposition}
\begin{proof}
``$\subseteq$'': Fix $\nu\in\B$ such that $X_\nu$ is finitely-representable in $\F$. Pick $v_1,\ldots,v_n\in V$ and $\varepsilon>0$. We shall show there is $\mu_0\in\F$ with $|\mu_0(v_i)-\nu(v_i)|<\varepsilon$, $i\leq n$. Let $m\in\Nat$ be such that $\{v_1,\ldots,v_n\}\subseteq \qeSpan\{e_j\setsep j\leq m\}$. Put $C:=\max\{\nu(v_i)\setsep i=1,\ldots,n\}$ and $Z:=\Span \{e_1,\ldots,e_m\}\subseteq X_\nu$. Since $X_\nu$ is finitely representable in $\F$, there is $\mu\in\F$ and a $(1+\tfrac{\varepsilon}{2C})$-isomorphism $T:Z\to X_\mu$. Set $x_i:=T(e_i)$, $i\leq m$, and extend $x_1,\ldots,x_m$ to a linearly independent sequence $(x_i)_{i=1}^{\infty}$ whose span is dense in $X_\mu$. Consider $\mu_0\in\PP$ given by setting $\mu_0(\sum_{i\in I} \alpha_i e_i)=\mu(\sum_{i\in I} \alpha_i x_i)$, where $I\subseteq \Nat$ is finite and $(\alpha_i)_{i\in I}\subseteq \Rat$. Clearly, $X_{\mu_0}\equiv X_\mu$ and $\mu_0\in\B$, so $\mu_0\in\F$. Finally, for every $i\leq n$ we have $v_i = \sum_{j=1}^m \alpha_j e_j$ for some $(\alpha_j)\in\Rea^m$ and so we have 
\[\mu_0(v_i) = \mu(\sum_{j=1}^m \alpha_j x_j)\leq (1+\tfrac{\varepsilon}{2C})\nu(\sum_{j=1}^m \alpha_j e_j) = (1+\tfrac{\varepsilon}{2C})\nu(v_i),\]
and similarly $\mu_0(v_i)\geq (1+\tfrac{\varepsilon}{2C})^{-1}\nu(v_i)\geq (1-\tfrac{\varepsilon}{2C})\nu(v_i)$.
Thus, $|\mu_0(v_i) - \nu(v_i)|\leq \tfrac{\varepsilon}{2}<\varepsilon$ for every $i\leq n$.

The case when we replace $\B$ with $\PP_\infty$ or $\PP$ is analogous, this time we only do not require $(x_i)_{i=1}^\infty$ to be linearly independent.
\medskip

\noindent``$\supseteq$'': Fix $\nu\in\overline{\mathcal{F}}\cap \B$. In order to see that $X_\nu$ is finitely representable in $\F$, we will use Lemma~\ref{lem:finiteReprLemma}. Pick $n\in\Nat$ and $\varepsilon>0$. Let $\phi_1$ be the function from Lemma~\ref{lem:approx}, let $\delta>0$ be such that $\phi_1(2\delta)<\varepsilon$ and let $N\subseteq V$ be a finite set which is $\delta$-dense for the sphere of $(\Span\{e_1,\ldots,e_n\},\nu)$ and $\nu(v)\in(1-\delta,1+\delta)$ for every $v\in N$. Pick $\mu\in\F$ such that $|\mu(v)-\nu(v)|<\delta$, $v\in N$. Then $id:(\Span\{e_1,\ldots,e_n\},\nu)\to (\Span\{e_1,\ldots,e_n\},\mu)$ is a $(1+\phi_1(2\delta))$-isomorphism. Thus, $X_\nu$ is finitely representable in $\F$. The case when we replace $\B$ with $\PP_\infty$ or $\PP$ is similar.
\end{proof}

This result has interesting consequences.

\begin{corollary}\label{cor:dense}
Let $X$ be a separable Banach space such that every Banach space is finitely representable in $X$. Then its isometry class is dense (in $\PP$, $\PP_\infty$ and also in $\B$).
\end{corollary}
\begin{proof}
Follows immediately from Proposition~\ref{prop:closureOfIsometryClasses}.
\end{proof}

\begin{corollary}\label{cor:Dvoretzky}
Let $X$ be a separable infinite-dimensional Banach space. Then $\isomtrclass[\B]{\ell_2}\subseteq \overline{\isomtrclass[\B]{X}}\cap\B$. The same holds if we replace $\B$ with $\PP_\infty$ or $\PP$.
\end{corollary}
\begin{proof}
By the Dvoretzky theorem, $\ell_2$ is finitely representable in every separable infinite-dimensional Banach space (see e.g. \cite[Theorem 13.3.7]{albiacKniha}). So we are done by applying Proposition~\ref{prop:closureOfIsometryClasses}.
\end{proof}

We conclude this subsection by showing another nice features of the above topologies on examples. We can show that the natural maps $K\mapsto C(K)$ and $\lambda\mapsto L_p(\lambda)$, where $K$ is a compact metrizable space and $\lambda$ is a Borel probability measure on a fixed compact metric space, are continuous.

\begin{example}\label{primitivnipriklad}
\begin{itemize}
\item[(a)] 
Let $\mathcal K([0,1]^\mathbb N)$ denote the space of all compact subsets of the Hilbert cube $[0,1]^\mathbb N$ endowed with the Vietoris topology. Then there exists a continuous mapping $\rho\colon\mathcal K([0,1]^\mathbb N)\rightarrow\mathcal P$ such that $X_{\rho(K)}\equiv C(K)$ for every $K\in\mathcal K([0,1]^\mathbb N)$.

\item[(b)] Let $L$ be a compact metric space, let $p\in[1,\infty)$ be fixed and let ${\mathcal Prob}(L)$ denote the space of all Borel probability measures on $L$ endowed with the weak* topology (generated by elements of the Banach space $C(L)$).
Then there exists a continuous mapping $\sigma\colon{\mathcal Prob}(L)\rightarrow\mathcal P$ such that $X_{\sigma(\lambda)}\equiv L_p(\lambda)$ for every $\lambda\in{\mathcal Prob}(L)$.
\end{itemize}
\end{example}

\begin{proof}
\begin{itemize}
\item[(a)] Let $\{f_i\colon i\in\mathbb N\}$ be a linearly dense subset of $C([0,1]^\mathbb N)$.
For every compact subset $K$ of $[0,1]^\mathbb N$, we define $\rho(K)\in\mathcal P$ by $$\rho(K)\left(\sum_{i=1}^nr_ie_i\right)=\sup_{x\in K}\left|\sum_{i=1}^nr_if_i(x)\right|, \qquad \sum_{i=1}^nr_ie_i\in V.$$
It is clear that $X_{\rho(K)}\equiv C(K)$, so we only need to check the continuity of $\rho$.
It is enough to show that $\rho^{-1}(U[v,I])$ is an open subset of $\mathcal K([0,1]^\mathbb N)$ for every $v\in V$ and every open interval $I$ (recall that $U[v,I]=\{\mu\in\mathcal P\colon\mu(v)\in I\}$).
So let us fix $\widetilde K\in\rho^{-1}(U[v,I])$, and assume that $v=\sum_{i=1}^nr_ie_i$.
Fix $x_0\in\widetilde K$ such that
$$\left|\sum_{i=1}^nr_if_i(x_0)\right|=\sup_{x\in\widetilde K}\left|\sum_{i=1}^nr_if_i(x)\right|.$$
Fix also $\varepsilon>0$ such that both numbers $\left|\sum_{i=1}^nr_if_i(x_0)\right|\pm\varepsilon$ belong to $I$.
Now find open subsets $U,V$ of $[0,1]^\mathbb N$ such that $x_0\in U$ and $\widetilde K\subseteq V$, and such that
$$\inf_{x\in U}\left|\sum_{i=1}^nr_if_i(x)\right|>\left|\sum_{i=1}^nr_if_i(x_0)\right|-\varepsilon$$
and
$$\sup_{x\in V}\left|\sum_{i=1}^nr_if_i(x)\right|<\sup_{x\in\widetilde K}\left|\sum_{i=1}^nr_if_i(x)\right|+\varepsilon.$$Then
$$\mathcal U:=\{K\in\mathcal K([0,1]^\mathbb N)\colon K\cap U\neq\emptyset\textnormal{ and }K\subseteq V\}$$
is an open neighborhood of $\widetilde K$ such that $\rho(\mathcal U)\subseteq U[v,I]$.
\item[(b)] This is similar to (a) but even easier. Let $\{g_i\colon i\in\mathbb N\}$ be a linearly dense subset of $C(L)$.
For every Borel probability measure $\lambda$ on $L$, we define $\sigma(\lambda)\in\mathcal P$ by $$\sigma(\lambda)\left(\sum_{i=1}^nr_ie_i\right)=\left(\int\limits_L\left|\sum_{i=1}^nr_ig_i\right|^p\,d\lambda\right)^{\frac 1p}, \qquad \sum_{i=1}^nr_ie_i\in V.$$
It is clear that $X_{\sigma(\lambda)}\equiv L_p(\lambda)$, so we only need to check the continuity of $\sigma$.
It is enough to show that $\sigma^{-1}(U[v,I])$ is an open subset of ${\mathcal Prob}(L)$ for every $v\in V$ and every open interval $I$.
But this is clear as, for $v=\sum_{i=1}^nr_ie_i$, we have
$$\sigma^{-1}(U[v,I])=\left\{\lambda\in{\mathcal Prob}(L)\colon\left(\int\limits_L\left|\sum_{i=1}^nr_ig_i\right|^p\,d\lambda\right)^{\frac 1p}\in I\right\}.$$
\end{itemize}
\end{proof}

\begin{remark}
After the introduction of the spaces $\PP$, $\PP_\infty$, and $\B$, one faces the question which of them is `the right one', with which to work. For now, we leave the question undecided. Since we are mainly interested in infinite-dimensional Banach spaces, we prefer to work mainly with $\PP_\infty$ and $\B$. On the other hand, it turns out that at least as far as one wants to transfer some computations performed in the space of pseudonorms directly to admissible topologies, the space $\PP$ is useful: Theorem~\ref{thm:reductionFromSBToP} below shows that whatever we compute in the space $\PP$ holds true also in any admissible topology.

Regarding $\PP_\infty$ and $\B$, in most of the arguments there is no difference whether we work with the former or the latter space. However, there are few exceptions when it seems to be convenient to work with the assumption that the sequence of vectors $\langle e_n\colon n\in\Nat\rangle\subseteq V$ is linearly independent, and then it might be more natural to work with $\B$.
\end{remark}

\section{Choice of the Polish space of separable Banach spaces}\label{sect:choice}

The main outcome of this section is Theorem~\ref{thm:Intro1} (denoted here as Theorem~\ref{thm:reductionFromBToSB}). We also prove partial converses to this result, see Theorem~\ref{thm:reductionFromSBToP} and Proposition~\ref{prop:reductionFromPInftyToB}. Let us give some more details.

\begin{enumerate}
    \item In the first subsection, we recall the coding $SB(X)$ (and $SB_\infty(X)$) of separable (infinite-dimensional) Banach spaces. We recall the notion of an admissible topology introduced in \cite{GS18}, which is a Polish topology corresponding to the Effros-Borel structure of $SB(X)$. We explore some basic relations between codings $\PP$, $\PP_\infty$, $\B$, $SB(X)$ and $SB_\infty(X)$. We show there is a continuous reduction from $SB(X)$ to $\PP$, a $\boldsymbol{\Sigma}_2^0$-measurable reduction from $\PP_\infty$ to $\B$, and a $\boldsymbol{\Sigma}_4^0$-measurable reduction from $\PP$ to $SB(X)$, see Theorem~\ref{thm:reductionFromSBToP}, Proposition~\ref{prop:reductionFromPInftyToB}, and Theorem~\ref{thm:reductionFromPToSBWorseVersion}. Here by a `reduction', we mean a map $\Phi$ such that a code and its image are both codes of the same (up to isometry) Banach space.
    \item The second subsection is devoted to the proof of Theorem~\ref{thm:reductionFromBToSB}, by which there is a $\boldsymbol{\Sigma}_2^0$-measurable reduction from $\B$ to $SB_\infty(X)$. Further, we note that the developed techniques also lead to a $\boldsymbol{\Sigma}_3^0$-measurable reduction from $\PP$ to $SB(X)$, which is an improvement of the result mentioned above.
\end{enumerate}

The importance of the reductions above is that there is not a big difference between Borel ranks when considered in any of the Polish spaces mentioned above.

Let us emphasize that the existence of a Borel reduction from $\B$ to $SB_\infty(X)$ has been essentially proved in \cite[Lemma 2.4]{Ku18}. Going through the proof of \cite[Lemma 2.4]{Ku18}, one may obtain a reduction which is $\boldsymbol{\Sigma}_3^0$-measurable; however, the proof does not seem to give a $\boldsymbol{\Sigma}_2^0$-measurable reduction (which is the optimal result). In order to obtain this improvement, see Theorem~\ref{thm:reductionFromBToSB}, we have to develop a whole machinery of new ideas in combination with very technical results, and this is the reason why we devote a whole subsection to the proof.

Since our reductions from $SB_\infty(X)$ to $\PP_\infty$ and from $\B$ to $SB_\infty(X)$ are optimal, it seems to be a very interesting open problem of whether there exists a continuous reduction from $\PP_\infty$ to $\B$ or at least a $\boldsymbol{\Sigma}_2^0$-measurable reduction from $\PP_\infty$ to $SB_\infty(X)$, see Question~\ref{q:1} and Question~\ref{q:2}.

\subsection{Relations between codings $\PP$, $\PP_\infty$, $\B$, $SB(X)$ and $SB_\infty(X)$}

Here we recall the approach to topologizing the class of all separable (infinite-dimensional) Banach spaces by Godefroy and Saint-Raymond from \cite{GS18} which was a partial motivation for our research.

\begin{definition}
Let $X$ be a Polish space and let us denote by $\F(X)$ the set of all closed subsets of $X$. For an open set $U\subseteq X$ we put $E^+(U) = \{F\in\F(X)\setsep U\cap F\neq \emptyset\}$. Following \cite{GS18}, we say that a Polish topology $\tau$ on the set $\F(X)$ is \emph{admissible} if it satisfies the following two conditions:
\begin{enumerate}[(i)]
    \item For every open subset $U$ of $X$, the set $E^+(U)$ is $\tau$-open.
    \item There exists a subbasis of $\tau$ such that every set from this subbasis is a countable union of sets of the form $E^+(U)\setminus E^+(V)$, where $U$ and $V$ are open in $X$.\footnote{note that the condition (ii) is different from what is mentioned in \cite{GS18}; however, as the authors have confirmed, there is a typo in the condition from \cite{GS18} which makes it wrong (otherwise no single one of the topologies mentioned in \cite{GS18} would be admissible)}
\end{enumerate}

We note that Godefroy and Saint-Raymond also suggest the following optional condition that is satisfied by many natural admissible topologies.
\begin{enumerate}[(i)]
\setcounter{enumi}{2}
    \item  The set $\{(x,F)\in X\times\F(X)\setsep x\in F\}$ is closed in $X\times\F(X)$.
\end{enumerate}
\medskip

If $X$ is a separable Banach space, we denote by $SB(X)\subseteq \F(X)$ the set of closed vector subspaces of $X$. We denote by $SB_\infty(X)$ the subset of $SB(X)$ consisting of infinite-dimensional spaces. We say that a topology on $SB(X)$ or $SB_\infty(X)$ is \emph{admissible} if it is induced by an admissible topology on $\F(X)$. Both $SB(X)$ and $SB_\infty(X)$ are Polish spaces when endowed with an admissible topology, see Remark~\ref{rem:polishTopologies}.

If $Z$ is a separable Banach space
, we put, similarly as in Notation~\ref{not:ismfrClass}, 
\[
\isomtrclass{Z}:=\{F\in SB(X)\setsep F\equiv Z\}\quad \text{and} \quad \isomrfclass{Z}:=\{F\in SB(X)\setsep F\simeq Z\}.
\]
It will be always clear from the context whether we work with subsets of $\PP$, or $SB(X)$.
\end{definition}

\begin{remark}\label{rem:polishTopologies}
If $X$ is a separable Banach space and $\tau$ is an admissible topology on $\F(X)$, then $SB(X)$ is a $G_\delta$-subset of $(\F(X),\tau)$ (see \cite[Section 3]{GS18}).
Moreover, by \cite[Corollary 4.2]{GS18}, $SB_\infty(X)$ is a $G_\delta$-subset of $(SB(X),\tau)$.
(In fact, the paper \cite{GS18} deals only with the case $X=C(2^\omega)$ but the generalization to any separable Banach space is easy.)
\end{remark}

A certain connection between codings $SB(X)$ and $\PP$ of separable Banach spaces might be deduced already from \cite{GS18}.

\begin{theorem}\label{thm:reductionFromSBToP}
Let $X$ be an isometrically universal separable Banach space and let $\tau$ be an admissible topology on $SB(X)$. Then there is a continuous mapping $\Phi:(SB(X),\tau)\to\PP$ such that for every $F\in SB(X)$ we have $F\equiv X_{\Phi(F)}$.
\end{theorem}
\begin{proof}
By \cite[Theorem 4.1]{GS18}, there are continuous functions $(f_n)_{n\in\Nat}$ on $SB(X)$ with values in $X$ such that for each $F\in SB(X)$ we have $\overline{\{f_n(F)\setsep n\in\Nat\}} = F$. Consider the mapping $\Phi$ given by $\Phi(F)(\sum_{n=1}^k a_n e_n):=\|\sum_{n=1}^k a_nf_n(F)\|_X$ for every $F\in SB(X)$ and $a_1,\ldots,a_k\in\Rat$. Then it is easy to see that $\Phi$ is the mapping we need.
\end{proof}

The following relation between various codings of Banach spaces as $SB(X)$ is easy.

\begin{observation}
Let $X,Y$ be isometrically universal separable Banach spaces and let $\tau_1$ and $\tau_2$ be admissible topologies on $SB(X)$ and $SB(Y)$, respectively. Then there is a $\boldsymbol{\Sigma}_2^0$-measurable mapping $f:(SB(X),\tau_1)\to (SB(Y),\tau_2)$ such that for every $F\in SB(X)$ we have $F\equiv f(F)$. Moreover, $f$ can be chosen such that for every open set $U\subseteq Y$ there is an open set $V\subseteq X$ such that $f^{-1}(E^+(U)) = E^+(V)$.
\end{observation}
\begin{proof}
Let $j:X\to Y$ be an isometry (not necessarily surjective). Then the mapping $f$ given by $f(F):= j(F)$, $F\in SB(X)$, does the job, because $f^{-1}(E^+(U)) = E^+(j^{-1}(U))$ for every open set $U\subseteq Y$.
\end{proof}

Let us note the following easy fact which we record here for a later reference. The proof is easy and so it is omitted.

\begin{lemma}\label{lem:condForSigmaDvaMeasurable}
Let $X$ be an isometrically universal separable Banach space, $\tau$ be an admissible topology on $SB(X)$, $Y$ be a Polish space, $f:Y\to SB(X)$ be a mapping and $n\in\Nat$, $n\geq 2$, be such that $f^{-1}(E^+(U))$ is a $\boldsymbol{\Delta}_n^0$ set in $Y$ for every open set $U\subseteq X$. Then $f$ is $\boldsymbol{\Sigma}_n^0$-measurable.
\end{lemma}

A straightforward idea leads to the following relation between $\PP_\infty$ and $\B$.

\begin{proposition}\label{prop:reductionFromPInftyToB}
There is a $\boldsymbol{\Sigma}_2^0$-measurable mapping $\Phi:\PP_\infty\to \B$ such that for every $\mu\in\PP_\infty$ we have $X_\mu\equiv X_{\Phi(\mu)}$.

Moreover, $\Phi$ can be chosen such that $\Phi^{-1}(U[v,I])\in \Delta_2^0(\PP_\infty)$ for each $v\in V$ and each open interval $I$.
\end{proposition}
\begin{proof}
For each $\mu\in\PP_\infty$ let us inductively define natural numbers $(n_k(\mu))_{k\in\Nat}$ by
\[\begin{split}
    n_1(\mu) & :=\min\{n\in\Nat\setsep \mu(e_n)\neq 0\},\\
    n_{k+1}(\mu) & :=\min\{n\in\Nat\setsep e_{n_1(\mu)},\ldots,e_{n_k(\mu)}, e_n\text{ are linearly independent}\}.
\end{split}
\]
Consider the mapping $\Phi$ given by $\Phi(\mu)(\sum_{i=1}^k a_n e_n):=\mu(\sum_{i=1}^k a_ie_{n_i(\mu)})$ for every $\mu\in \PP_\infty$ and $a_1,\ldots,a_k\in\Rat$. It is easy to see that $\Phi(\mu)\in \B$ and that $X_\mu$ is isometric to $X_{\Phi(\mu)}$ for each $\mu\in\PP_\infty$.

For all natural numbers $N_1<\ldots<N_k$ the set $\{\mu\in\PP_\infty\setsep n_1(\mu)=N_1,\ldots,n_k(\mu)=N_k\}$ is a $\Delta_2^0$ set in $\PP_\infty$. Indeed, we may prove it by induction on $k$ because for each $k\in\Nat$ and each $\mu\in\PP_\infty$ we have that $n_1(\mu)=N_1,\ldots,n_{k+1}(\mu)=N_{k+1}$ iff
\[\begin{split}
n_1(\mu) & =N_1, \ldots, n_k(\mu)=N_{k}\quad\&\\
&\forall n=N_k+1,\ldots,N_{k+1}-1: e_{N_1},\ldots,e_{N_k},e_n\text{ are linearly dependent}\\
& \&\quad e_{N_1},\ldots,e_{N_{k+1}}\text{ are linearly independent},
\end{split}\]
which is an intersection of a $\Delta_2^0$-condition (by the inductive assumption) with a closed and an open condition (by Lemma~\ref{lem:infiniteDimIsGDelta}).

Let us pick $v = \sum_{i=1}^k a_n e_n\in V$ and an open interval $I$. Then
\[
\Phi^{-1}(U[v,I]) = \{\mu\in\PP_\infty\setsep \mu(\sum_{i=1}^k a_ie_{n_i(\mu)})\in I\},
\]
which is a $\boldsymbol{\Delta}_2^0$ set in $\PP_\infty$. Indeed, on one hand we have $\mu\in\Phi^{-1}(U[v,I])$ iff there are natural numbers $N_1<N_2<\ldots<N_k$ such that $n_1(\mu)=N_1,\ldots,n_k(\mu)=N_k$ and $\mu(\sum_{i=1}^k a_i e_{N_i})\in I$, which witnesses that $\Phi^{-1}(U[v,I])\in\boldsymbol{\Sigma}_2^0(\PP_\infty)$ as it is a countable union of $\Delta_2^0$ sets. On the other hand, we have that $\mu\in\Phi^{-1}(U[v,I])$ iff for each $l\in\Nat$ we have that either $n_k(\mu) > l$ or there are natural numbers $N_1<N_2<\ldots<N_k\leq l$ such that $n_1(\mu)=N_1,\ldots,n_k(\mu)=N_k$ and $\mu(\sum_{i=1}^k a_i e_{N_i})\in I$, which witnesses that $\Phi^{-1}(U[v,I])\in\boldsymbol{\Pi}_2^0(\PP_\infty)$ as it is a countable intersection of $\Delta_2^0$ sets.

This proves the ``Moreover'' part from which it easily follows that $\Phi$ is $\boldsymbol{\Sigma}_2^0$-measurable.
\end{proof}

\begin{remark}\label{rem:finiteDim}
For $d\in\Nat$, let us consider the sets $\PP_d:=\{\mu\in\PP\setsep \dim X_\mu = d\}$ and
\[
    \B_d:=\{\mu\in\PP\setsep e_1,\ldots,e_d \text{ is a basis of $X_\mu$ and $\mu(e_i)=0$ for every $i>d$}\}.   
\]
A similar argument as in Proposition~\ref{prop:reductionFromPInftyToB} shows that for every $d\in\Nat$ there is a $\boldsymbol{\Sigma}_2^0$-measurable mapping $\Phi:\PP_d\to \B_d$ such that for every $\mu\in\PP_d$ we have $X_\mu\equiv X_{\Phi(\mu)}$.
\end{remark}

Finally let us consider the reduction from $\PP$ to $SB(X)$. An optimal result would be to have a $\boldsymbol{\Sigma}_2^0$-reduction. This is because, as was already observed in \cite{GS18}, the identity map between two admissible topologies is only $\Sigma^0_2$-measurable in general. Using the ideas of the proof of \cite[Lemma 2.4]{Ku18} we obtain Theorem~\ref{thm:reductionFromPToSBWorseVersion}. This result is improved in the next subsection, see Theorem~\ref{thm:reductionFromPToSB}, but since some steps remain the same, let us give a sketch of the argument (we will be a bit sketchy at the places which will be modified later).

\begin{lemma}\label{lem:sufficientForReduction}
Let $n\in\Nat$, $X$ be an isometrically universal separable Banach space and let $\tau$ be an admissible topology on $SB(X)$. Let there exist $\boldsymbol{\Sigma}_n^0$-measurable mappings $\chi_{k} : \mathcal{B} \to X$, $k\in\Nat$, such that $X_\mu\equiv \overline{\Span}\{\chi_{k}(\mu)\setsep k\in\Nat\}$ for every $\mu\in\B$.

Then there exists a $\boldsymbol{\Sigma}_{n+1}^0$-measurable mapping $\Phi:\B\to (SB(X),\tau)$ such that for every $\mu\in\B$ we have $X_\mu\equiv \Phi(\mu)$.
\end{lemma}
\begin{proof}
Consider the mapping $\Phi:\B\to (SB(X),\tau)$ defined as \[\Phi(\nu):=\overline{\Span}\{\chi_{k}(\nu)\setsep k\in\Nat\},\;\;\; \nu\in\B.\]
We have $X_{\nu}\equiv \Phi(\nu)$. For every open set $U\subseteq X$, using the $\boldsymbol{\Sigma}_n^0$-measurability of $\chi_{k}$'s, it is easy to see that $\Phi^{-1}(E^+(U))$ is a $\boldsymbol{\Sigma}_n^0$-set in $\B$. Thus, by Lemma~\ref{lem:condForSigmaDvaMeasurable}, the mapping $\Phi$ is $\boldsymbol{\Sigma}_{n+1}^0$-measurable.
\end{proof}

\begin{remark}
Similarly as in Remark~\ref{rem:finiteDim}, an analogous approach leads to a similar statement valid for any $\B_d$, $d\in\Nat$, instead of $\B$.
\end{remark}

\begin{theorem}\label{thm:reductionFromPToSBWorseVersion}
Let $X$ be an isometrically universal separable Banach space and let $\tau$ be an admissible topology on $SB(X)$. Then there is a $\boldsymbol{\Sigma}_4^0$-measurable mapping $\Phi:\PP\to (SB(X),\tau)$ such that for every $\mu\in\PP$ we have $X_\mu\equiv \Phi(\mu)$.
\end{theorem}
\begin{proof}[Sketch of the proof]By Remark~\ref{rem:finiteDim}, it suffices to find, for every $d\in\Nat\cup\{\infty\}$, a $\boldsymbol{\Sigma}_3^0$-measurable reduction from $\B_d$ to $SB(X)$, where $\B_\infty = \B$. This is done for every $d\in\Nat\cup\{\infty\}$ in a similar way. Let us concentrate further only on the case of $d=\infty$, the other cases are similar. From the proof of \cite[Lemma 2.4]{Ku18}, it follows that there are Borel measurable mappings $\chi_{k} : \mathcal{B} \to X$, $k\in\Nat$, such that $X_\mu\equiv \closedSpan{\chi_{k}(\mu)\setsep k\in\Nat}$ for every $\mu\in\B$. A careful inspection of the proof actually shows that the mappings $\chi_k$ are $\boldsymbol{\Sigma}_2^0$-measurable (since this part is improved in the next subsection, see Proposition~\ref{propspojvnor}, we do not give any more details here). Thus, an application of Lemma~\ref{lem:sufficientForReduction} finishes the proof.
\end{proof}

\subsection{An optimal reduction from $\mathcal B$ to $SB(X)$}
The last subsection is devoted to the proof of the following result. The rest of the paper does not depend on it.
\begin{theorem}\label{thm:reductionFromBToSB}
Let $X$ be an isometrically universal separable Banach space and let $\tau$ be an admissible topology on $SB(X)$. Then there is a $\boldsymbol{\Sigma}_2^0$-measurable mapping $\Phi:\B\to (SB(X),\tau)$ such that for every $\mu\in\B$ we have $X_\mu\equiv \Phi(\mu)$.
\end{theorem}

The main ingredient of the proof is the following.
\begin{proposition} \label{propspojvnor}
	For any isometrically universal separable Banach space $ X $, there exist continuous mappings $ \chi_{k} : \mathcal{B} \to X, k \in \mathbb{N}, $ such that
	$$ \Big\Vert \sum_{k=1}^{n} a_{k} \chi_{k}(\nu) \Big\Vert = \nu \Big( \sum_{k=1}^{n} a_{k} e_{k} \Big) $$
	for every $ \sum_{k=1}^{n} a_{k} e_{k} \in c_{00} $ and every $ \nu \in \mathcal{B} $.
\end{proposition}

\begin{remark}\label{rem:finiteDimProspojvnor}
Similarly as in Remark~\ref{rem:finiteDim}, we may easily obtain a variant of Proposition~\ref{propspojvnor} for $\B_d$, $d\in\Nat$. Indeed, let $d\in\Nat$ be given. For $\nu\in\B_d$, let us define $\widetilde{\nu}\in\B$ by
\[\widetilde{\nu} \Big( \sum_{i=1}^\infty a_ie_i \Big) :=\nu \Big( \sum_{i=1}^d a_ie_i \Big) + \sum_{i=d+1}^\infty |a_i|,\quad \sum_{i=1}^\infty a_ie_i\in c_{00}.\]
If $ \chi_{k}$, $k\in\Nat,$ are as in Proposition~\ref{propspojvnor}, then we may consider mappings $ \widetilde{\chi}_{k}:\B_d\to X$, $k\leq d,$ defined by $\widetilde{\chi}_{k}(\nu) = \chi_{k}(\widetilde{\nu}) $, $ \nu \in \mathcal{B}_{d} $.
\end{remark}

We postpone the proof of Proposition~\ref{propspojvnor} to the very end of this subsection. 
\begin{proof}[Proof of Theorem~\ref{thm:reductionFromBToSB}]
Follows immediately from Lemma~\ref{lem:sufficientForReduction} and Proposition~\ref{propspojvnor}.
\end{proof}

Similarly as above, we obtain also the following.
\begin{theorem}\label{thm:reductionFromPToSB}
Let $X$ be an isometrically universal separable Banach space and let $\tau$ be an admissible topology on $SB(X)$. Then there is a $\boldsymbol{\Sigma}_3^0$-measurable mapping $\Phi:\PP\to (SB(X),\tau)$ such that for every $\mu\in\PP$ we have $X_\mu\equiv \Phi(\mu)$.
\end{theorem}
\begin{proof}
This is similar to the proof of Theorem~\ref{thm:reductionFromPToSBWorseVersion}, the only modification is that we use Proposition~\ref{propspojvnor} and Remark~\ref{rem:finiteDimProspojvnor} instead of the reference to the proof of \cite[Lemma 2.4]{Ku18}.
\end{proof}

The aim of the remainder of this subsection is now to prove Proposition~\ref{propspojvnor}.
Its proof is based on two auxiliary results, namely Proposition~\ref{propspojvnorcharakt} and Lemma~\ref{lemmspojvnor}, the first of them is a reformulation of our task. Although we will use only the implication $ (2') \Rightarrow (1) $ for the set of norms with $ \nu(e_{k}) = 1 $, we prove a bit more and keep the formulation for a general set of (pseudo)norms. The implication $ (1) \Rightarrow (2) $, which is not obligatory for us, is a simple application of the Hahn-Banach theorem, and we think that it is appropriate to include the proof.

In fact, we do not know if the conditions hold for the set of all pseudonorms (or, equivalently, for the set of pseudonorms with $ \nu(e_{k}) \leq 1 $). So, it is open if Proposition~\ref{propspojvnor} holds not only for $ \mathcal{B} $, but even for $ \mathcal{P} $.

\begin{proposition} \label{propspojvnorcharakt}
	For each $ \mathcal{A} \subseteq \mathcal{P} $, the following conditions are equivalent:

	(1) There are a separable Banach space $ U $ and continuous mappings $ \chi_{k} : \mathcal{A} \to U, k \in \mathbb{N}, $ such that
	$$ \Big\Vert \sum_{k=1}^{n} a_{k} \chi_{k}(\nu) \Big\Vert = \nu \Big( \sum_{k=1}^{n} a_{k} e_{k} \Big) $$
	for every $ \sum_{k=1}^{n} a_{k} e_{k} \in c_{00} $ and every $ \nu \in \mathcal{A} $.

	(2) There are continuous functions $ \alpha_{k} : \mathcal{A}^{2} \to [0, \infty), k \in \mathbb{N}, $ such that
	$$ \alpha_{k}(\nu, \nu) = 0 $$
	for every $ \nu $ and $ k $ and the following property is satisfied: If $ \nu \in \mathcal{A} $ and $ z^{*} \in (c_{00})^{\#} $ satisfy $ |z^{*}(x)| \leq \nu(x) $ for every $ x \in c_{00} $, then there is a mapping $ \Gamma : \mathcal{A} \to (c_{00})^{\#} $ such that $ \Gamma(\nu) = z^{*} $, $ |\Gamma(\mu)(x)| \leq \mu(x) $ for every $ \mu \in \mathcal{A} $ and $ x \in c_{00} $, and
	$$ |\Gamma(\mu)(e_{k}) - \Gamma(\lambda)(e_{k})| \leq \alpha_{k}(\mu, \lambda) $$
	for every $ \mu, \lambda \in \mathcal{A} $ and every $ k $.

	Moreover, if $ \mathcal{A} $ consists only of pseudonorms $ \nu $ with $ \nu(e_{k}) \leq 1 $ for every $ k $, then these conditions are equivalent with:

	(2') For every $ \eta \in [0, 1) $, there are continuous functions $ \beta_{k} : \mathcal{A}^{2} \to [0, \infty), k \in \mathbb{N}, $ such that
	$$ \beta_{k}(\nu, \nu) = 0 $$
	for every $ \nu $ and $ k $ and the following property is satisfied: If $ \nu \in \mathcal{A} $ and $ z^{*} \in (c_{00})^{\#} $ satisfy $ |z^{*}(x)| \leq \nu(x) $ for every $ x \in c_{00} $, then there is a mapping $ \Gamma : \mathcal{A} \to (c_{00})^{\#} $ such that $ \Gamma(\nu) = \eta \cdot z^{*} $, $ |\Gamma(\mu)(x)| \leq \mu(x) $ for every $ \mu \in \mathcal{A} $ and $ x \in c_{00} $, and
	$$ |\Gamma(\mu)(e_{k}) - \Gamma(\lambda)(e_{k})| \leq \beta_{k}(\mu, \lambda) $$
	for every $ \mu, \lambda \in \mathcal{A} $ and every $ k $.
\end{proposition}

\begin{remark}
	The conditions (1) and (2) from Proposition~\ref{propspojvnorcharakt} are equivalent also with the following one:

	\emph{(3) There are continuous functions $ \alpha_{k} : \mathcal{A}^{2} \to [0, \infty), k \in \mathbb{N}, $ such that
		$$ \alpha_{k}(\nu, \nu) = 0 $$
		for every $ \nu $ and $ k $ and
		$$ \sum_{\mu, \lambda, k} |a_{\mu, \lambda, k}| \alpha_{k}(\mu, \lambda) \geq \nu \Big( \sum_{\lambda, k} (a_{\nu, \lambda, k} - a_{\lambda, \nu, k}) e_{k} \Big) - \sum_{\mu \neq \nu} \mu \Big( \sum_{\lambda, k} (a_{\mu, \lambda, k} - a_{\lambda, \mu, k}) e_{k} \Big) $$
		for every $ \nu \in \mathcal{A} $ and every system $ (a_{\mu, \lambda, k})_{\mu, \lambda \in \mathcal{A}, k \in \mathbb{N}} $ of real numbers with finite support.}
		
		The proof is similar to the proof of $(1)\Leftrightarrow(2)$ below (for $(1)\Rightarrow (3)$ the choice of $\alpha_k$'s is the same as in the proof of $(1)\Rightarrow (2)$, for $(3)\Rightarrow (1)$ the construction of the space $U$ is the same as in $(2)\Rightarrow (1)$). We omit the full proof, because the details are technical and we do not use condition $(3)$ any further. Let us note that even though we tried to find an application of condition $(3)$, we did not find it and this is basically the reason why we had to develop conditions $(2)$ and $(2')$.
\end{remark}

\begin{proof}[Proof of Proposition~\ref{propspojvnorcharakt}]
	(1) $ \Rightarrow $ (2): Given such $ U $ and $ \chi_{k} : \mathcal{A} \to U, k \in \mathbb{N} $, we put
	$$ \alpha_{k}(\nu, \mu) = \Vert \chi_{k}(\nu) - \chi_{k}(\mu) \Vert, \quad \nu, \mu \in \mathcal{A}, k \in \mathbb{N}. $$
	Denote by $ I_{\mu} : (c_{00}, \mu) \to U $ the isometry given by $ e_{k} \mapsto \chi_{k}(\mu) $. Let $ \nu \in \mathcal{A} $ and $ z^{*} \in (c_{00})^{\#} $ satisfying $ |z^{*}(x)| \leq \nu(x) $ be given. For $ x, y \in c_{00} $ with $ I_{\nu}x = I_{\nu}y $, we have $ |z^{*}(y - x)| \leq \nu(y - x) = \Vert I_{\nu}(y - x) \Vert = 0 $, and so $ z^{*}(x) = z^{*}(y) $. Thus, the formula
	$$ u^{*}(I_{\nu}x) = z^{*}(x), \quad x \in c_{00}, $$
	defines a functional on $ I_{\nu}(c_{00}) $ such that $ |u^{*}(I_{\nu}x)| = |z^{*}(x)| \leq \nu(x) = \Vert I_{\nu}x \Vert $. By the Hahn-Banach theorem, we can extend $ u^{*} $ to the whole $ U $ in the way that
	$$ |u^{*}(u)| \leq \Vert u \Vert, \quad u \in U. $$
	For every $ \mu \in \mathcal{A} $, let us put
	$$ \Gamma(\mu)(x) = u^{*}(I_{\mu}x), \quad x \in c_{00}. $$
	We obtain $ |\Gamma(\mu)(x)| = |u^{*}(I_{\mu}x)| \leq \Vert I_{\mu}x \Vert = \mu(x) $ and $ |\Gamma(\mu)(e_{k}) - \Gamma(\lambda)(e_{k})| = |u^{*}(I_{\mu}e_{k}) - u^{*}(I_{\lambda}e_{k})| = |u^{*}(\chi_{k}(\mu) - \chi_{k}(\lambda))| \leq \Vert \chi_{k}(\mu) - \chi_{k}(\lambda) \Vert = \alpha_{k}(\mu, \lambda) $ for $ \mu, \lambda \in \mathcal{A} $ and $ k \in \mathbb{N} $.

	(2) $ \Rightarrow $ (1): Given such $ \alpha_{k} : \mathcal{A}^{2} \to [0, \infty), k \in \mathbb{N} $, we define a subset of $ c_{00}(\mathcal{A} \times \mathbb{N}) $ by
	$$ \Omega = \mathrm{co} \bigg( \bigcup_{\mu} \Big\{ \sum_{k} a_{k} e_{\mu, k} : \mu \Big( \sum_{k} a_{k} e_{k} \Big) \leq 1 \Big\} \cup \bigcup_{\mu, \lambda, k} \big\{ c \cdot (e_{\mu, k} - e_{\lambda, k}) : |c| \cdot \alpha_{k}(\mu, \lambda) \leq 1 \big\} \bigg), $$
	and denote the corresponding Minkowski functional by $ \varrho $. Let $ U $ be the completion of the quotient space $ X/N $, where $ X = (c_{00}(\mathcal{A} \times \mathbb{N}), \varrho) $ and $ N = \{ x \in c_{00}(\mathcal{A} \times \mathbb{N}) : \rho(x) = 0 \} $. In what follows, we identify every $ x \in c_{00}(\mathcal{A} \times \mathbb{N}) $ with its equivalence class $ [x]_{N} \in U $. Let us define
	$$ \chi_{k} : \mathcal{A} \to U, \quad \nu \mapsto e_{\nu, k}. $$
	As $ c \cdot (e_{\mu, k} - e_{\lambda, k}) \in \Omega $ whenever $ |c| \cdot \alpha_{k}(\mu, \lambda) \leq 1 $, we obtain $ \varrho(e_{\mu, k} - e_{\lambda, k}) \leq \alpha_{k}(\mu, \lambda) $, that is, $ \varrho(\chi_{k}(\mu) - \chi_{k}(\lambda)) \leq \alpha_{k}(\mu, \lambda) $. For a fixed $ \mu $, we have $ \alpha_{k}(\mu, \lambda) \to \alpha_{k}(\mu, \mu) = 0 $ as $ \lambda \to \mu $, and consequently $ \varrho(\chi_{k}(\mu) - \chi_{k}(\lambda)) \to 0 $ as $ \lambda \to \mu $. Therefore, $ \chi_{k} $ is continuous on $ \mathcal{A} $. It follows that the image of $ \chi_{k} $ is separable. As these images contain all basic vectors $ e_{\nu, k} $, the space $ U $ is separable.

	We need to show that
	$$ \nu(x) = \varrho(\overline{x}) $$
	for fixed $ \nu \in \mathcal{A} $, $ x = \sum_{k \in \mathbb{N}} a_{k} e_{k} \in c_{00} $ and its image $ \overline{x} = \sum_{k \in \mathbb{N}} a_{k} e_{\nu, k} $. The inequality $ \nu(x) \geq \varrho(\overline{x}) $ follows immediately from the definition of $ \Omega $ (for any $ c \geq \nu(x) $ with $ c > 0 $, we have $ \nu(\frac{1}{c}x) \leq 1 $, and so $ \frac{1}{c}\overline{x} \in \Omega $, hence $ \varrho(\frac{1}{c}\overline{x}) \leq 1 $ and $ \varrho(\overline{x}) \leq c $). Let us show the opposite inequality $ \nu(x) \leq \varrho(\overline{x}) $. Using the Hahn-Banach theorem, we can pick $ z^{*} \in (c_{00})^{\#} $ satisfying $ z^{*}(x) = \nu(x) $ and $ |z^{*}(y)| \leq \nu(y) $ for every $ y \in c_{00} $. 	Let $ \Gamma : \mathcal{A} \to (c_{00})^{\#} $ be the mapping provided for $ \nu $ and $ z^{*} $, and let $ u^{*} \in (c_{00}(\mathcal{A} \times \mathbb{N}))^{\#} $ be given by
	$$ u^{*}(e_{\mu, k}) = \Gamma(\mu)(e_{k}), \quad \mu \in \mathcal{A}, k \in \mathbb{N}. $$
	Then $ u^{*}(\overline{x}) = u^{*}(\sum_{k} a_{k} e_{\nu, k}) = \sum_{k} a_{k} u^{*}(e_{\nu, k}) = \sum_{k} a_{k} \Gamma(\nu)(e_{k}) = \Gamma(\nu)(\sum_{k} a_{k} e_{k}) = z^{*}(x) = \nu(x) $. It is sufficient to show that $ u^{*} \leq 1 $ on $ \Omega $ (equivalently $ |u^{*}(y)| \leq \varrho(y) $ for every $ y \in c_{00}(\mathcal{A} \times \mathbb{N}) $), since it follows that $ \nu(x) = u^{*}(\overline{x}) \leq \varrho(\overline{x}) $.

	To show that $ u^{*} \leq 1 $ on $ \Omega $, we need to check that
	$$ \mu \Big( \sum_{k} b_{k} e_{k} \Big) \leq 1 \quad \Rightarrow \quad u^{*} \Big( \sum_{k} b_{k} e_{\mu, k} \Big) \leq 1 $$
	and
	$$ |c| \cdot \alpha_{k}(\mu, \lambda) \leq 1 \quad \Rightarrow \quad u^{*} \big( c \cdot (e_{\mu, k} - e_{\lambda, k}) \big) \leq 1. $$
	Concerning the first implication, we compute $ u^{*}(\sum_{k} b_{k} e_{\mu, k}) = \sum_{k} b_{k} u^{*}(e_{\mu, k}) = \sum_{k} b_{k} \Gamma(\mu)(e_{k}) = \Gamma(\mu)(\sum_{k} b_{k} e_{k}) \leq \mu (\sum_{k} b_{k} e_{k}) \leq 1 $. Concerning the second implication, we compute $ u^{*}(c \cdot (e_{\mu, k} - e_{\lambda, k})) = c u^{*}(e_{\mu, k}) - c u^{*}(e_{\lambda, k}) = c \Gamma(\mu)(e_{k}) - c \Gamma(\lambda)(e_{k}) \leq |c| \alpha_{k}(\mu, \lambda) \leq 1 $.

(2) $ \Rightarrow $ (2'): The choice $ \beta_{k} = \alpha_{k} $ works. Indeed, if $ \Gamma $ is provided by (2), we can take $ \eta \cdot \Gamma $.

(2') $ \Rightarrow $ (2): For every $ n \in \mathbb{N} $, let $ \beta_{k}^{n} : \mathcal{A}^{2} \to [0, \infty), k \in \mathbb{N}, $ be provided by (2') for $ \eta = (1 - 2^{-n}) $. We can assume that each $ \beta_{k}^{n} $ is a pseudometric. Indeed, instead of $ \beta_{k}^{n} $, we can take the maximal minorizing pseudometric $ \widetilde{\beta}^{n}_{k} $ (in such a case, $ \widetilde{\beta}^{n}_{k} $ is continuous and, since the function $ (\mu, \lambda) \mapsto |\Gamma(\mu)(e_{k}) - \Gamma(\lambda)(e_{k})| $ is a pseudometric, if it minorizes $ \beta_{k}^{n} $, then it minorizes $ \widetilde{\beta}^{n}_{k} $ as well). Moreover, we can assume that $ \beta_{1}^{1} $ is a metric (it is possible to add a compatible metric on $ \mathcal{A} $ to $ \beta_{1}^{1} $).

Let us define
$$ \alpha(\mu, \lambda) = \max_{n, k} \min \{ \beta_{k}^{n}(\mu, \lambda), 2^{- \max\{n, k\}} \}, \quad \mu, \lambda \in \mathcal{A}. $$
It is easy to check that $ \alpha $ is continuous. Due to our additional assumptions, $ \alpha $ is a metric on $ \mathcal{A} $. We want to show that there are some constants $ c_{k} $ such that the choice $ \alpha_{k} = c_{k} \cdot \alpha $ works.

Let $ \nu \in \mathcal{A} $ and $ z^{*} \in (c_{00})^{\#} $ satisfy $ |z^{*}(x)| \leq \nu(x) $ for every $ x \in c_{00} $. For every $ n \in \mathbb{N} $, there is a mapping $ \Gamma^{n} : \mathcal{A} \to (c_{00})^{\#} $ such that
$$ \Gamma^{n}(\nu) = (1 - 2^{-n}) \cdot z^{*}, $$
$$ |\Gamma^{n}(\mu)(x)| \leq \mu(x), \quad \mu \in \mathcal{A}, \; x \in c_{00}, $$
and, if we denote
$$ \gamma_{k}^{n}(\mu) = \Gamma^{n}(\mu)(e_{k}), $$
then
$$ |\gamma_{k}^{n}(\mu) - \gamma_{k}^{n}(\lambda)| \leq \beta_{k}^{n}(\mu, \lambda) $$
for every $ \mu, \lambda \in \mathcal{A} $ and every $ k $. Let us note that
$$ |\gamma_{k}^{n}(\mu)| \leq 1, $$
as $ |\gamma_{k}^{n}(\mu)| = |\Gamma^{n}(\mu)(e_{k})| \leq \mu(e_{k}) \leq 1 $ by the assumption on $ \mathcal{A} $.

Now, we define the desired mapping $ \Gamma $. For practical purposes, we first define
$$ \gamma_{k}^{n}(\mu) = 0 \quad \textnormal{for } n \in \mathbb{Z}, n \leq 0. $$
For every $ n \in \mathbb{Z} $, let $ f_{n} $ denote the piecewise linear function supported by $ [2^{-n-3}, 2^{-n-1}] $ which is linear on $ [2^{-n-3}, 2^{-n-2}] $ and $ [2^{-n-2}, 2^{-n-1}] $, and for which $ f_{n}(2^{-n-2}) = 1 $. In this way, we have $ \sum_{n \in \mathbb{Z}} f_{n} = 1 $ on $ (0, \infty) $. We define
$$ \gamma_{k}(\nu) = z^{*}(e_{k}) $$
and
$$ \gamma_{k}(\mu) = \sum_{n \in \mathbb{Z}} f_{n}(\alpha(\mu, \nu)) \gamma_{k}^{n}(\mu), \quad \mu \neq \nu. $$
Finally, we put $ \Gamma(\mu)(e_{k}) = \gamma_{k}(\mu) $, so $ \Gamma(\nu) = z^{*} $ and $ \Gamma(\mu) = \sum_{n \in \mathbb{N}} f_{n}(\alpha(\mu, \nu)) \Gamma^{n}(\mu) $ for $ \mu \neq \nu $. In both cases $ \mu = \nu $ and $ \mu \neq \nu $, it follows that $ |\Gamma(\mu)(x)| \leq \mu(x) $ for every $ x \in c_{00} $. It remains to prove the inequality
$$ |\gamma_{k}(\mu) - \gamma_{k}(\lambda)| \leq c_{k} \cdot \alpha(\mu, \lambda) $$
for some suitable constants $ c_{k} $.

Let us show that the implication
\begin{equation} \label{impl1}
\alpha(\mu, \lambda) < 2^{-n} \quad \Rightarrow \quad |\gamma_{k}^{n}(\mu) - \gamma_{k}^{n}(\lambda)| \leq 2^{k+1} \alpha(\mu, \lambda)
\end{equation}
holds. Clearly, we can suppose that $ n \geq 1 $. If $ \alpha(\mu, \lambda) \geq 2^{-k} $, then $ 2^{k+1} \alpha(\mu, \lambda) \geq 2 \geq |\gamma_{k}^{n}(\mu) - \gamma_{k}^{n}(\lambda)| $. So, let us assume that $ \alpha(\mu, \lambda) < 2^{-k} $. Since $ \min \{ \beta_{k}^{n}(\mu, \lambda), 2^{- \max\{n, k\}} \} \leq \alpha(\mu, \lambda) < 2^{- \max\{n, k\}} $, we have $ \min \{ \beta_{k}^{n}(\mu, \lambda), 2^{- \max\{n, k\}} \} = \beta_{k}^{n}(\mu, \lambda) $, and so $ |\gamma_{k}^{n}(\mu) - \gamma_{k}^{n}(\lambda)| \leq \beta_{k}^{n}(\mu, \lambda) = \min \{ \beta_{k}^{n}(\mu, \lambda), 2^{- \max\{n, k\}} \} \leq \alpha(\mu, \lambda) \leq 2^{k+1} \alpha(\mu, \lambda) $.

Next, we show that
\begin{equation} \label{impl2}
2^{-n-4} \leq \alpha(\mu, \nu) < 2^{-n} \quad \Rightarrow \quad |\gamma_{k}^{n}(\mu) - \gamma_{k}(\nu)| \leq (2^{k+1} + 16) \alpha(\mu, \nu).
\end{equation}
If $ n \geq 1 $, then $ \gamma_{k}(\nu) - \gamma_{k}^{n}(\nu) = z^{*}(e_{k}) - (1 - 2^{-n}) z^{*}(e_{k}) = 2^{-n} z^{*}(e_{k}) $. If $ n \leq 0 $, then $ \gamma_{k}(\nu) - \gamma_{k}^{n}(\nu) = z^{*}(e_{k}) $. In both cases, $ |\gamma_{k}(\nu) - \gamma_{k}^{n}(\nu)| \leq 2^{-n} |z^{*}(e_{k})| \leq 2^{-n} \nu(e_{k}) \leq 2^{-n} \leq 2^{4} \alpha(\mu, \nu) $. Using \eqref{impl1}, we can compute
$$ |\gamma_{k}^{n}(\mu) - \gamma_{k}(\nu)| \leq |\gamma_{k}^{n}(\mu) - \gamma_{k}^{n}(\nu)| + |\gamma_{k}^{n}(\nu) - \gamma_{k}(\nu)| \leq (2^{k+1} + 2^{4}) \alpha(\mu, \nu). $$

Further, it follows from \eqref{impl2} that
\begin{equation} \label{ineq3}
|\gamma_{k}(\mu) - \gamma_{k}(\nu)| \leq (2^{k+1} + 16) \alpha(\mu, \nu).
\end{equation}
Indeed, since $ f_{n} $ is supported by $ [2^{-n-3}, 2^{-n-1}] $, we have always
$$ f_{n}(\alpha(\mu, \nu)) |\gamma_{k}^{n}(\mu) - \gamma_{k}(\nu)| \leq f_{n}(\alpha(\mu, \nu)) (2^{k+1} + 16) \alpha(\mu, \nu), $$
and it is sufficient to use that $ \gamma_{k}(\mu) - \gamma_{k}(\nu) = \sum_{n \in \mathbb{Z}} f_{n}(\alpha(\mu, \nu)) (\gamma_{k}^{n}(\mu) - \gamma_{k}(\nu)) $ for $ \mu \neq \nu $.

Now, we are going to investigate the value $ |\gamma_{k}(\mu) - \gamma_{k}(\lambda)| $. First, we have
\begin{equation} \label{impl4}
\alpha(\lambda, \nu) \geq 2 \alpha(\mu, \nu) \quad \Rightarrow \quad |\gamma_{k}(\mu) - \gamma_{k}(\lambda)| \leq 3 \cdot (2^{k+1} + 16) \alpha(\mu, \lambda).
\end{equation}
Indeed, as $ \alpha(\mu, \lambda) \geq \alpha(\lambda, \nu) - \alpha(\mu, \nu) \geq 2 \alpha(\mu, \nu) - \alpha(\mu, \nu) = \alpha(\mu, \nu) $, we can apply \eqref{ineq3} and write
\begin{align*}
|\gamma_{k}(\mu) - \gamma_{k}(\lambda)| & \leq |\gamma_{k}(\mu) - \gamma_{k}(\nu)| + |\gamma_{k}(\lambda) - \gamma_{k}(\nu)| \\
 & \leq (2^{k+1} + 16) (\alpha(\mu, \nu) + \alpha(\lambda, \nu)) \\
 & = (2^{k+1} + 16) (\alpha(\lambda, \nu) - \alpha(\mu, \nu) + 2 \alpha(\mu, \nu)) \\
 & \leq (2^{k+1} + 16)(1 + 2) \alpha(\mu, \lambda).
\end{align*}

Now, we prove the last but the most challenging implication
\begin{equation} \label{impl5}
\alpha(\mu, \nu) \leq \alpha(\lambda, \nu) < 2 \alpha(\mu, \nu) \quad \Rightarrow \quad |\gamma_{k}(\mu) - \gamma_{k}(\lambda)| \leq [12 (2^{k+1} + 16) + 2^{k+1}] \alpha(\mu, \lambda).
\end{equation}

Let us compute
\begin{align*}
 & \gamma_{k}(\mu) - \gamma_{k}(\lambda) = \sum_{n \in \mathbb{Z}} \big[ f_{n}(\alpha(\mu, \nu)) \gamma_{k}^{n}(\mu) - f_{n}(\alpha(\lambda, \nu)) \gamma_{k}^{n}(\lambda) \big] \\
 & = \sum_{n \in \mathbb{Z}} \big[ f_{n}(\alpha(\mu, \nu)) \gamma_{k}^{n}(\mu) - f_{n}(\alpha(\lambda, \nu)) \gamma_{k}^{n}(\mu) + f_{n}(\alpha(\lambda, \nu)) \gamma_{k}^{n}(\mu) - f_{n}(\alpha(\lambda, \nu)) \gamma_{k}^{n}(\lambda) \big] \\
 & = \sum_{n \in \mathbb{Z}} [f_{n}(\alpha(\mu, \nu)) - f_{n}(\alpha(\lambda, \nu))] \gamma_{k}^{n}(\mu) + \sum_{n \in \mathbb{Z}} f_{n}(\alpha(\lambda, \nu)) [\gamma_{k}^{n}(\mu) - \gamma_{k}^{n}(\lambda)] \\
 & = \sum_{n \in \mathbb{Z}} [f_{n}(\alpha(\mu, \nu)) - f_{n}(\alpha(\lambda, \nu))] (\gamma_{k}^{n}(\mu) - \gamma_{k}(\nu)) + \sum_{n \in \mathbb{Z}} f_{n}(\alpha(\lambda, \nu)) [\gamma_{k}^{n}(\mu) - \gamma_{k}^{n}(\lambda)].
\end{align*}
Hence, $ |\gamma_{k}(\mu) - \gamma_{k}(\lambda)| $ is less than or equal to
$$ \sum_{n \in \mathbb{Z}} |f_{n}(\alpha(\mu, \nu)) - f_{n}(\alpha(\lambda, \nu))| |\gamma_{k}^{n}(\mu) - \gamma_{k}(\nu)| + \sum_{n \in \mathbb{Z}} f_{n}(\alpha(\lambda, \nu)) |\gamma_{k}^{n}(\mu) - \gamma_{k}^{n}(\lambda)|. $$

Let us notice that
\begin{itemize}
\item $ f_{n}(\alpha(\mu, \nu)) \neq 0 $ iff $ 2^{-n-3} < \alpha(\mu, \nu) < 2^{-n-1} $,
\item $ f_{n}(\alpha(\lambda, \nu)) \neq 0 $ iff $ 2^{-n-3} < \alpha(\lambda, \nu) < 2^{-n-1} $, and $ 2^{-n-4} < \alpha(\mu, \nu) < 2^{-n-1} $ in this case,
\item the function $ f_{n} $ is Lipschitz with the constant $ 2^{n+3} $.
\end{itemize}
So, if $ f_{n}(\alpha(\mu, \nu)) \neq 0 $ or $ f_{n}(\alpha(\lambda, \nu)) \neq 0 $, then $ 2^{-n-4} < \alpha(\mu, \nu) < 2^{-n-1} $, and \eqref{impl2} can be applied. We obtain for the first sum that
\begin{align*}
\sum_{n \in \mathbb{Z}} |f_{n}(\alpha(\mu, \nu)) & - f_{n}(\alpha(\lambda, \nu))| |\gamma_{k}^{n}(\mu) - \gamma_{k}(\nu)| \\
 & \leq \sum_{2^{-n-4} < \alpha(\mu, \nu) < 2^{-n-1}} 2^{n+3} |\alpha(\mu, \nu) - \alpha(\lambda, \nu)| (2^{k+1} + 16) \alpha(\mu, \nu) \\
 & \leq \sum_{2^{-n-4} < \alpha(\mu, \nu) < 2^{-n-1}} 2^{n+3} \alpha(\mu, \lambda) (2^{k+1} + 16) 2^{-n-1} \\
 & \leq 3 \cdot 2^{2} (2^{k+1} + 16) \alpha(\mu, \lambda).
\end{align*}

Concerning the second sum, we notice that if $ f_{n}(\alpha(\lambda, \nu)) \neq 0 $, then $ \alpha(\lambda, \nu) < 2^{-n-1} $, and so $ \alpha(\mu, \lambda) \leq \alpha(\mu, \nu) + \alpha(\lambda, \nu) \leq 2\alpha(\lambda, \nu) < 2^{-n} $. Applying \eqref{impl1}, we obtain
$$ \sum_{n \in \mathbb{Z}} f_{n}(\alpha(\lambda, \nu)) |\gamma_{k}^{n}(\mu) - \gamma_{k}^{n}(\lambda)| \leq \sum_{n \in \mathbb{Z}} f_{n}(\alpha(\lambda, \nu)) 2^{k+1} \alpha(\mu, \lambda) = 2^{k+1} \alpha(\mu, \lambda), $$
and \eqref{impl5} follows.

Finally, we finish the proof with the observation that \eqref{impl4} and \eqref{impl5} provide
$$ |\gamma_{k}(\mu) - \gamma_{k}(\lambda)| \leq [12 (2^{k+1} + 16) + 2^{k+1}] \alpha(\mu, \lambda). $$
We can suppose that $ \alpha(\mu, \nu) \leq \alpha(\lambda, \nu) $. If $ \alpha(\lambda, \nu) \geq 2 \alpha(\mu, \nu) $, we use \eqref{impl4}, and if $ \alpha(\lambda, \nu) < 2 \alpha(\mu, \nu) $, we use \eqref{impl5}.
\end{proof}

\begin{definition}
Let $ \mathcal{B}_{(1)} $ denote the set of all norms $ \nu \in \mathcal{B} $ such that $ \nu(e_{k}) = 1 $ for each $ k \in \mathbb{N} $.
\end{definition}

Now we introduce the second auxiliary result for proving Proposition~\ref{propspojvnor}.

\begin{lemma} \label{lemmspojvnor}
The condition (2') from Proposition~\ref{propspojvnorcharakt} is valid for $ \mathcal{A} = \mathcal{B}_{(1)} $.
\end{lemma}

The proof of the lemma is organized according to the order of quantifiers in condition (2'). This is perhaps not the most natural order for reading, and the reader may skip the first paragraph of the proof, in which the functions $ \beta_{k} $ are introduced in a somewhat incomprehensible way, and come back later. The next part, in which functions $ \gamma_{k} $ (the coordinates of $ \Gamma $) are constructed, is natural in some sense. In a recursive procedure, functionals are extended to a space with dimension $ 1 $ bigger, so there are similarities with the proof of the Hanh-Banach theorem. In our situation, the upper bound by the corresponding norm is a bit relaxed in each step (see formula \eqref{cond2}), which is allowed since we prove the relaxed condition (2'). This is needed later for having a control over the difference $ |\gamma_{k}(\mu) - \gamma_{k}(\lambda)| $.

\begin{proof}
Let $ \eta \in [0, 1) $ be given. We fix numbers $ \kappa_{k} < 1 $ such that $ \eta \leq \kappa_{1} < \kappa_{2} < \kappa_{3} < \dots $. For every $ \mu, \lambda \in \mathcal{B}_{(1)} $, we define recursively
$$ \beta_{1}(\mu, \lambda) = 0 $$
and
\begin{align*}
\beta_{k+1}(\mu, \lambda) = \sup \bigg\{ & \Big| \mu \Big( e_{k+1} + \sum_{i=1}^{k} a_{i} e_{i} \Big) - \lambda \Big( e_{k+1} + \sum_{i=1}^{k} a_{i} e_{i} \Big) \Big| + \sum_{i=1}^{k} |a_{i}| \beta_{i}(\mu, \lambda) : \\
 & a_{1}, \dots, a_{k} \in \mathbb{R}, \; \min \Big\{ \mu \Big( \sum_{i=1}^{k} a_{i} e_{i} \Big), \lambda \Big( \sum_{i=1}^{k} a_{i} e_{i} \Big) \Big\} < \frac{2\kappa_{k+1}}{\kappa_{k+1} - \kappa_{k}} \bigg\} .
\end{align*}
Clearly, $ \beta_{k}(\nu, \nu) = 0 $ for every $ \nu \in \mathcal{B}_{(1)} $. Let us sketch a proof of continuity of the functions $ \beta_{k} $. The function $ \beta_{1} = 0 $ is obviously continuous. Assuming that $ \beta_{i} $ is continuous for every $ i \leq k $, we consider for $ \delta > 0 $ the set
$$ \mathcal{U}_{\delta}^{k+1}(\mu) = \bigg\{ \mu' \in \mathcal{B}_{(1)} : \Big( \forall x \in \mathrm{span} \{ e_{1}, \dots, e_{k+1} \} \setminus \{ 0 \} : (1 + \delta)^{-1} < \frac{\mu'(x)}{\mu(x)} < 1 + \delta \Big) \bigg\}. $$
Using Lemma~\ref{lem:approx}, it is easy to see that $ \mathcal{U}_{\delta}^{k+1}(\mu) $ is an open neighborhood of $ \mu $ in $ \mathcal{B}_{(1)} $. Given $ \varepsilon > 0 $, we can find $ \delta > 0 $ such that $ |\beta_{k+1}(\mu', \lambda') - \beta_{k+1}(\mu, \lambda)| < \varepsilon $ for every $ (\mu', \lambda') \in \mathcal{U}_{\delta}^{k+1}(\mu) \times \mathcal{U}_{\delta}^{k+1}(\lambda) $. The details are left to the reader.

Let us prove that the functions $ \beta_{k} $ work. Given $ \nu \in \mathcal{B}_{(1)} $ and $ z^{*} \in (c_{00})^{\#} $ satisfying $ |z^{*}(x)| \leq \nu(x) $ for every $ x \in c_{00} $, we define first
$$ \gamma_{1}(\mu) = \eta \cdot z^{*}(e_{1}), \quad \mu \in \mathcal{B}_{(1)}. $$
Recursively, we define for every $ k \in \mathbb{N} $ functions
$$ u_{k+1}(\mu) = \sup_{a_{1}, \dots, a_{k}} \Big[ - \kappa_{k+1} \mu \Big( - e_{k+1} + \sum_{i=1}^{k} a_{i} e_{i} \Big) + \sum_{i=1}^{k} a_{i} \gamma_{i}(\mu) \Big], $$
$$ v_{k+1}(\mu) = \inf_{a_{1}, \dots, a_{k}} \Big[ \kappa_{k+1} \mu \Big( e_{k+1} + \sum_{i=1}^{k} a_{i} e_{i} \Big) - \sum_{i=1}^{k} a_{i} \gamma_{i}(\mu) \Big] $$
and
$$ \gamma_{k+1}(\mu) = p_{k+1} u_{k+1}(\mu) + q_{k+1} v_{k+1}(\mu), $$
where numbers $ p_{k+1} \geq 0, q_{k+1} \geq 0 $ with $ p_{k+1} + q_{k+1} = 1 $ are chosen in the way that
\begin{equation} \label{cond1}
\gamma_{k+1}(\nu) = \eta \cdot z^{*}(e_{k+1}).
\end{equation}
Let us check that it is possible to choose such numbers. Note first that $ \gamma_{1}(\nu) = \eta \cdot z^{*}(e_{1}) $. Assuming that the functions $ \gamma_{i} $ are already defined and satisfy $ \gamma_{i}(\nu) = \eta \cdot z^{*}(e_{i}) $ for $ i \leq k $, we notice that, for every $ a_{1}, \dots, a_{k} \in \mathbb{R} $,
$$ \pm \eta \cdot z^{*}(e_{k+1}) + \sum_{i=1}^{k} a_{i} \gamma_{i}(\nu) = \eta \cdot z^{*} \Big( \pm e_{k+1} + \sum_{i=1}^{k} a_{i} e_{i} \Big) \leq \kappa_{k+1} \nu \Big( \pm e_{k+1} + \sum_{i=1}^{k} a_{i} e_{i} \Big), $$
and consequently
$$ \eta \cdot z^{*}(e_{k+1}) \geq - \kappa_{k+1} \nu \Big( - e_{k+1} + \sum_{i=1}^{k} a_{i} e_{i} \Big) + \sum_{i=1}^{k} a_{i} \gamma_{i}(\nu), $$
$$ \eta \cdot z^{*}(e_{k+1}) \leq \kappa_{k+1} \nu \Big( e_{k+1} + \sum_{i=1}^{k} a_{i} e_{i} \Big) - \sum_{i=1}^{k} a_{i} \gamma_{i}(\nu). $$
This gives
$$ u_{k+1}(\nu) \leq \eta \cdot z^{*}(e_{k+1}) \leq v_{k+1}(\nu), $$
and it follows that suitable $ p_{k+1} $ and $ q_{k+1} $ do exist.

Let us prove that
\begin{equation} \label{cond2}
\sum_{i=1}^{k} a_{i} \gamma_{i}(\mu) \leq \kappa_{k} \mu \Big( \sum_{i=1}^{k} a_{i} e_{i} \Big)
\end{equation}
for every $ \mu \in \mathcal{B}_{(1)} $, $ k \in \mathbb{N} $ and $ a_{1}, \dots, a_{k} \in \mathbb{R} $. For $ k = 1 $, we just write $ a_{1} \gamma_{1}(\mu) = a_{1} \eta \cdot z^{*}(e_{1}) \leq |a_{1}| \kappa_{1} \nu(e_{1}) = |a_{1}| \kappa_{1} = |a_{1}| \kappa_{1} \mu(e_{1}) = \kappa_{1} \mu(a_{1} e_{1}) $. Assume that \eqref{cond2} is valid for $ k $. We show first that
$$ u_{k+1}(\mu) \leq \gamma_{k+1}(\mu) \leq v_{k+1}(\mu). $$
Clearly, it is sufficient to show just that $ u_{k+1}(\mu) \leq v_{k+1}(\mu) $. Given $ b_{1}, \dots, b_{k} $ and $ c_{1}, \dots, c_{k} $, we need to check that
$$ - \kappa_{k+1} \mu \Big( - e_{k+1} + \sum_{i=1}^{k} b_{i} e_{i} \Big) + \sum_{i=1}^{k} b_{i} \gamma_{i}(\mu) \leq \kappa_{k+1} \mu \Big( e_{k+1} + \sum_{i=1}^{k} c_{i} e_{i} \Big) - \sum_{i=1}^{k} c_{i} \gamma_{i}(\mu). $$
But this is easy, as
\begin{align*}
\sum_{i=1}^{k} b_{i} \gamma_{i}(\mu) + \sum_{i=1}^{k} c_{i} \gamma_{i}(\mu) & = \sum_{i=1}^{k} (b_{i} + c_{i}) \gamma_{i}(\mu) \leq \kappa_{k} \mu \Big( \sum_{i=1}^{k} (b_{i} +c_{i}) e_{i} \Big) \\
 & \leq \kappa_{k+1} \Big[ \mu \Big( - e_{k+1} + \sum_{i=1}^{k} b_{i} e_{i} \Big) + \mu \Big( e_{k+1} + \sum_{i=1}^{k} c_{i} e_{i} \Big) \Big].
\end{align*}
Now, let us verify \eqref{cond2} for $ k + 1 $. We can suppose that $ a_{k+1} = \pm 1 $. For $ a_{k+1} = 1 $, it is enough to use
$$ \gamma_{k+1}(\mu) \leq v_{k+1}(\mu) \leq \kappa_{k+1} \mu \Big( e_{k+1} + \sum_{i=1}^{k} a_{i} e_{i} \Big) - \sum_{i=1}^{k} a_{i} \gamma_{i}(\mu), $$
and for $ a_{k+1} = -1 $, it is enough to use
$$ \gamma_{k+1}(\mu) \geq u_{k+1}(\mu) \geq - \kappa_{k+1} \mu \Big( - e_{k+1} + \sum_{i=1}^{k} a_{i} e_{i} \Big) + \sum_{i=1}^{k} a_{i} \gamma_{i}(\mu). $$

Next, let us prove that
\begin{equation} \label{cond3}
|\gamma_{k}(\mu) - \gamma_{k}(\lambda)| \leq \beta_{k}(\mu, \lambda)
\end{equation}
for every $ \mu, \lambda \in \mathcal{B}_{(1)} $ and $ k \in \mathbb{N} $. This is clear for $ k = 1 $, as $ \gamma_{1} $ is constant. Assume that \eqref{cond3} is valid for $ i \leq k $. To prove it for $ k + 1 $, it is sufficient to show the inequalities
$$ |u_{k+1}(\mu) - u_{k+1}(\lambda)| \leq \beta_{k+1}(\mu, \lambda) \quad \textnormal{and} \quad |v_{k+1}(\mu) - v_{k+1}(\lambda)| \leq \beta_{k+1}(\mu, \lambda). $$
We consider the function $ v_{k+1} $ only, since the inequality for $ u_{k+1} $ can be shown in the same way. Let us note first that, in the definition of $ v_{k+1}(\mu) $, it is possible to take the infimum only over $ k $-tuples with
$$ \mu \Big( \sum_{i=1}^{k} a_{i} e_{i} \Big) < \frac{2\kappa_{k+1}}{\kappa_{k+1} - \kappa_{k}}. $$
Indeed, for $ a_{1}, \dots, a_{k} $ which do not satisfy this condition, using \eqref{cond2}, we obtain
\begin{align*}
 \kappa_{k+1} \mu \Big( e_{k+1} + & \sum_{i=1}^{k} a_{i} e_{i} \Big) - \sum_{i=1}^{k} a_{i} \gamma_{i}(\mu) \\
 & \geq \kappa_{k+1} \mu \Big( \sum_{i=1}^{k} a_{i} e_{i} \Big) - \kappa_{k+1} \mu(e_{k+1}) - \kappa_{k} \mu \Big( \sum_{i=1}^{k} a_{i} e_{i} \Big) \\
 & = (\kappa_{k+1} - \kappa_{k}) \mu \Big( \sum_{i=1}^{k} a_{i} e_{i} \Big) - \kappa_{k+1} \\
 & \geq 2 \kappa_{k+1} - \kappa_{k+1} = \kappa_{k+1} = \kappa_{k+1} \mu \Big( e_{k+1} + \sum_{i=1}^{k} 0 \cdot e_{i} \Big) - \sum_{i=1}^{k} 0 \cdot \gamma_{i}(\mu).
\end{align*}
Now, \eqref{cond3} is provided by the following computation, in which every sup/inf is meant over $ k $-tuples with $ \mu \Big( \sum_{i=1}^{k} a_{i} e_{i} \Big) < \frac{2\kappa_{k+1}}{\kappa_{k+1} - \kappa_{k}} $ or $ \lambda \Big( \sum_{i=1}^{k} a_{i} e_{i} \Big) < \frac{2\kappa_{k+1}}{\kappa_{k+1} - \kappa_{k}} $:
\begin{align*}
 & |v_{k+1}(\mu) - v_{k+1}(\lambda)| \\
 & = \bigg| \inf \Big[ \kappa_{k+1} \mu \Big( e_{k+1} + \sum_{i=1}^{k} a_{i} e_{i} \Big) - \sum_{i=1}^{k} a_{i} \gamma_{i}(\mu) \Big] \\
 & \hspace{4cm} - \inf \Big[ \kappa_{k+1} \lambda \Big( e_{k+1} + \sum_{i=1}^{k} a_{i} e_{i} \Big) - \sum_{i=1}^{k} a_{i} \gamma_{i}(\lambda) \Big] \bigg| \\
 & \leq \sup \bigg| \Big[ \kappa_{k+1} \mu \Big( e_{k+1} + \sum_{i=1}^{k} a_{i} e_{i} \Big) - \sum_{i=1}^{k} a_{i} \gamma_{i}(\mu) \Big] \\
 & \hspace{4cm} - \Big[ \kappa_{k+1} \lambda \Big( e_{k+1} + \sum_{i=1}^{k} a_{i} e_{i} \Big) - \sum_{i=1}^{k} a_{i} \gamma_{i}(\lambda) \Big] \bigg| \\
 & = \sup \bigg| \kappa_{k+1} \Big[ \mu \Big( e_{k+1} + \sum_{i=1}^{k} a_{i} e_{i} \Big) - \lambda \Big( e_{k+1} + \sum_{i=1}^{k} a_{i} e_{i} \Big) \Big] - \sum_{i=1}^{k} a_{i} \big( \gamma_{i}(\mu) - \gamma_{i}(\lambda) \big) \bigg| \\
 & \leq \sup \bigg[ \Big| \mu \Big( e_{k+1} + \sum_{i=1}^{k} a_{i} e_{i} \Big) - \lambda \Big( e_{k+1} + \sum_{i=1}^{k} a_{i} e_{i} \Big) \Big| + \sum_{i=1}^{k} |a_{i}| \beta_{i}(\mu, \lambda) \bigg] = \beta_{k+1}(\mu, \lambda).
\end{align*}

Finally, as usual, we put $ \Gamma(\mu)(e_{k}) = \gamma_{k}(\mu) $. The required properties of $ \Gamma $ follow now from \eqref{cond1}, \eqref{cond2} and \eqref{cond3}. Thus, the functions $ \beta_{k} $ work, and the proof of the lemma is completed.
\end{proof}

\begin{proof}[Proof of Proposition~\ref{propspojvnor}]
Let $ X $ be an isometrically universal separable Banach space. By Lemma~\ref{lemmspojvnor}, the condition (2') from Proposition~\ref{propspojvnorcharakt} is valid for $ \mathcal{A} = \mathcal{B}_{(1)} $. Hence, the condition (1) from this proposition is valid for  $ \mathcal{A} = \mathcal{B}_{(1)} $ as well. There are a separable Banach space $ U $ and continuous mappings $ \chi_{k} : \mathcal{B}_{(1)} \to U, k \in \mathbb{N}, $ such that
$$ \Big\Vert \sum_{k=1}^{n} a_{k} \chi_{k}(\nu) \Big\Vert = \nu \Big( \sum_{k=1}^{n} a_{k} e_{k} \Big) $$
for every $ \sum_{k=1}^{n} a_{k} e_{k} \in c_{00} $ and every $ \nu \in \mathcal{B}_{(1)} $. Since $ X $ contains an isometric copy of $ U $, we can suppose that $ U \subseteq X $.

Let us consider the continuous mapping $ \Psi : \mathcal{B} \mathcal \to \mathcal{B}_{(1)} $ given by
$$ \Psi(\mu) \Big( \sum_{k=1}^{n} a_{k} e_{k} \Big) = \mu \Big( \sum_{k=1}^{n} \frac{a_{k}}{\mu(e_{k})} e_{k} \Big). $$
If we define
$$ \widetilde{\chi}_{k}(\mu) = \mu(e_{k}) \cdot \chi_{k}(\Psi(\mu)), \quad \mu \in \mathcal{B}, $$
for each $ k \in \mathbb{N} $, then we get
$$ \Big\Vert \sum_{k=1}^{n} b_{k} \widetilde{\chi}_{k}(\mu) \Big\Vert = \Big\Vert \sum_{k=1}^{n} b_{k} \mu(e_{k}) \chi_{k}(\Psi(\mu)) \Big\Vert = \Psi(\mu) \Big( \sum_{k=1}^{n} b_{k} \mu(e_{k}) e_{k} \Big) = \mu \Big( \sum_{k=1}^{n} b_{k} e_{k} \Big) $$
for every $ \sum_{k=1}^{n} b_{k} e_{k} \in c_{00} $ and every $ \mu \in \mathcal{B} $.
\end{proof}

To summarize, we obtained an optimal reduction from $\B$ to $SB_\infty(X)$ and from $SB_\infty(X)$ to $\PP_\infty$. However, our reduction from $\PP_\infty$ to $\B$ seems not to be optimal, so one is tempted to ask the following.

\begin{question}\label{q:1}
Does there exist a continuous mapping $\Phi:\PP_\infty\to\B$ such that for every $\mu\in\PP_\infty$ we have $X_\mu\equiv X_{\Phi(\mu)}$?
\end{question}

Note that a positive answer to Question~\ref{q:1} would imply a positive answer to Question~\ref{q:2} and that a sufficient condition for a positive solution of Question~\ref{q:2} is provided by Proposition~\ref{propspojvnorcharakt}.

\begin{question}\label{q:2}
Let $X$ be an isometrically universal separable Banach space and let $\tau$ be an admissible topology on $SB(X)$. Does there exist a $\boldsymbol{\Sigma}_2^0$-measurable mapping $\Phi:\PP_\infty \to (SB(X),\tau)$ such that for every $\mu\in\PP_\infty$ we have $X_\mu\equiv \Phi(\mu)$?
\end{question}
\section{Generic properties}\label{sect:generic}
As soon as one has a Polish space, or more generally a Baire space, of some objects it is natural and often useful to find properties (of these objects) that are generic; that is, the corresponding subset of the space is comeager. In the case of the spaces $\PP$, $\PP_\infty$ and $\B$ we resolve this problem completely, see Theorem~\ref{thm:gurariiTypicalInP}  which is a more precise description of the content of Theorem~\ref{thm:Intro4}.

In the case of the spaces $SB(X)$ with an admissible topology, this is the content of Problem 5.5 from \cite{GS18}. We show that in that case the situation is more complicated. In particular, we confirm the suspicion of Godefroy and Saint-Raymond that being meager in $SB(X)$ is not independent of the chosen admissible topology, see Theorem~\ref{thm:GurariiinWijsman} and  Theorem \ref{thm:Baireadmiss2} which together imply Theorem~\ref{thm:Intro5}.

\subsection{Generic objects in $\PP$}The main result of this subsection is Theorem~\ref{thm:gurariiTypicalInP} below. Although it may not appear surprising to a specialist in Fra\" iss\' e theory, as the Gurari\u{\i} space is the Fra\" iss\' e limit of finite-dimensional Banach spaces, see e.g. \cite[Theorem 4.3]{Ya15}, its proof is far more involved than analogous results for countable Fra\" iss\' e limits. Also, the result has actually several applications in estimating the complexities of isometry and isomorphism classes of other Banach spaces, see \cite{CDDK21}, where this issues is addressed.

\begin{theorem}\label{thm:gurariiTypicalInP}
Let $\mathbb{G}$ be the Gurari\u{\i} space. Then the isometry class $\isomtrclass[\mathcal I]{\mathbb{G}}$ is a dense $G_\delta$-set in $\mathcal{I}$ for any $\mathcal{I}\in\{\PP,\PP_\infty,\B\}$.
\end{theorem}

Let us recall what the Gurari\u{\i} space is. One of the characterizations of the Gurari\u{\i} space is the following, for more details we refer the interested reader e.g. to \cite{CGK} (the characterization below is provided by \cite[Lemma 2.2]{CGK}).

\begin{definition}\label{def:Gurarii}
The Gurari\u{\i} space is the unique (up to isometry) separable Banach space such that  for every $\varepsilon > 0$ and every isometric embedding $g:A\to B$, where $B$ is a finite-dimensional Banach space and $A$ is a subspace of $\mathbb{G}$, there is a $(1+\varepsilon)$-isomorphism $f:B\to \mathbb{G}$ such that $\|f\circ g - id_A\|\leq \varepsilon$.
\end{definition}
 
In the remainder of this subsection we prove Theorem~\ref{thm:gurariiTypicalInP}. Let us start with the most technical part, namely that $\isomtrclass[\PP_\infty]{\mathbb{G}}$ is a $G_\delta$-set in $\PP_\infty$.

We need two technical lemmas first.

\begin{lemma}\label{lem:approx2}
\begin{enumerate}[(i)]
    \item\label{it:approxIsometryOnBasis} Given a basis $\b_E = \{e_1,\ldots,e_n\}$ of a finite-dimensional Banach space $E$, there is $C>0$ and a function $\phi_2^{\b_E}:[0,C)\to[0,\infty)$ continuous at zero with $\phi_2^{\b_E}(0)=0$ such that whenever $X$ is a Banach space with $E\subseteq X$ and $\{x_i\setsep i\leq n\}\subseteq X$ are such that $\|x_i-e_i\|<\varepsilon$, $i\leq n$, for some $\varepsilon<C$, then the linear operator $T:E\to X$ given by $T(e_i):=x_i$ is a $(1+\phi_2^{\b_E}(\varepsilon))$-isomorphism and $\|T-Id_E\|\leq \phi_2^{\b_E}(\varepsilon)$.
	\item\label{it:approxNety} Let $\varepsilon\in (0,1)$, $T:X\to Y$ be a surjective $(1+\varepsilon)$-isomorphism between Banach spaces $X$ and $Y$, $N$ be $\varepsilon$-dense for $S_X$. Then $T(N)$ is $3\varepsilon$-dense for $S_Y$.
\end{enumerate}
\end{lemma}
\begin{proof}
\ref{it:approxIsometryOnBasis}: Pick $C>0$ such that $C\sum_{i=1}^n|\lambda_i|\leq \|\sum_{i=1}^n \lambda_i e_i\|$ for every $(\lambda_i)_{i=1}^n\in \Rea^n$. Then for any $x = \sum_{i=1}^n \lambda_i e_i$ we have
\[\|Tx - x\|\leq \sum_{i=1}^n |\lambda_i| \|x_i-e_i\| < \frac{\varepsilon}{C} \|x\|.\]
Thus, $\|T-Id_E\|<\frac{\varepsilon}{C}$, $\|T\|\leq 1+\frac{\varepsilon}{C}$ and $\|Tx\|\geq (1-\frac{\varepsilon}{C})\|x\|= (1+\frac{\varepsilon}{C-\varepsilon})^{-1}\|x\|$. Thus, we may put $\phi_2^{\mathfrak{b}_E}(\varepsilon):=\frac{\varepsilon}{C-\varepsilon}$ for $\varepsilon\in[0,C)$.\\
\ref{it:approxNety}: Let $\varepsilon>0$, $T:X\to Y$ and $N$ be as in the assumptions. Then for every $y\in S_Y$ there is $x\in N$ with $\|x-\frac{T^{-1}(y)}{\|T^{-1}(y)\|}\|<\varepsilon$. Thus, we have
\[\begin{split}
	\|y - Tx\|&\leq \Big\|y - \frac{y}{\|T^{-1}(y)\|}\Big\| + \|T\|\cdot\Big\|x-\frac{T^{-1}(y)}{\|T^{-1}(y)\|}\Big\|\\
	& < \Big|1 - \frac{1}{\|T^{-1}y\|}\Big| + (1+\varepsilon)\varepsilon\leq \varepsilon + 2\varepsilon = 3\varepsilon.
\end{split}\]
\end{proof}

\begin{lemma}\label{lem:BIsDense}
For every $\mu\in\PP$, finite set $A\subseteq V$ and $\varepsilon>0$ there exists $\nu\in\B$ with $|\mu(x)-\nu(x)|<\varepsilon$ and $\nu(x)\in \Rat$ for every $x\in A$.
\end{lemma}
\begin{proof}
It suffices to define such a norm $\nu$ on $\Span A$ since then we can easily find some extension to the whole $V$. We assume that $0\notin A$ and moreover we can assume that no two elements of $A$ lie in the same one-dimensional subspace, i.e. are scalar multiples of each other. Indeed, otherwise we would find a subset $A'\subseteq A$ where no elements are scalar multiples of each other and every element of $A$ is a scalar multiple, necessarily rational scalar multiple, of some element from $A'$. Then proving the fact for $A'$ for sufficiently small $\delta$ automatically proves it for $A$ and $\varepsilon$.

We enumerate $A$ as $\{a_1,\ldots,a_n\}$ and so that the first $k$ elements $a_1,\ldots,a_k$, for some $k\leq n$, are linearly independent and form a basis of $\Span A$.

\noindent {\bf Claim.} By perturbing $\mu$ on $A$ by an arbitrarily small $\delta>0$ we can without loss of generality assume that for every $i\leq n$, $\mu(a_i)< K_i:=\inf\{\sum_{j\in J}\mu(\alpha_j a_j)\colon i\notin J\subseteq \{1,\ldots,n\}, a_i=\sum_{j\in J} \alpha_j a_j\}$.

Suppose the claim is proved. Then for every $i\leq n$ we set $\nu'(a_i)$ to be an arbitrary positive rational number in the interval $\left[\mu(a_i),\min\{K_i,\mu(a_i)+\varepsilon\}\right)$. From the assumption it is now clear that for all $i\leq n$, we have $$\nu'(a_i)\leq \inf\{\sum_{j=1}^n |\alpha_j|\nu'(a_j)\colon a_i=\sum_{j=1}^n \alpha_j a_j\}.$$
We extend $\nu'$ to a norm $\nu$ on $\Span A$ by the formula $$\nu(v):=\inf\{\sum_{i=1}^n |\alpha_i| \nu'(a_i)\colon v=\sum_{i=1}^n \alpha_i a_i\},$$ for $v\in\Span A$. From the previous assumption, it follows that $\nu(a_i)=\nu'(a_i)$, for all $i\leq n$. Moreover, $\nu$ is indeed a norm since $\nu(a_i)>0$ for all $i\leq n$, and the infimum in the definition of $\nu$ is, by compactness, always attained.

It remains to prove the claim. Let $\|\cdot\|_2$ be the $\ell_2$ norm on $\Span A$ with $a_1,\ldots,a_k$ the orthonormal basis. For each $m\in\Nat$ set $\mu_m:=\mu+\frac{\|\cdot\|_2}{m}$. Clearly $\mu_m\to\mu$, so it suffices to show that each $\mu_m$ satisfies the condition from the claim. Suppose that for some $m\in\Nat$ and $i\leq n$ we have $$\mu_m(a_i)=\inf\{\sum_{j\in J}\mu_m(\alpha_j a_j)\colon i\notin J\subseteq \{1,\ldots,n\}, a_i=\sum_{j\in J} \alpha_j a_j\}.$$ By compactness, the infimum is attained, i.e. there exists $(\alpha_j)_{j\leq n}$, with $\alpha_i=0$, $a_i=\sum_{j\leq n} \alpha_j a_j$ and $\mu_m(a_i)=\sum_{j\leq n} \mu_m(\alpha_j a_j)$. Indeed, if the infimum is approximated by a sequence $(\alpha^l_1,\ldots,\alpha^l_n)_{l\in\Nat}\subseteq \Rea^n$, then since each coordinate is bounded (because up to finitely many $l$'s we have $\sum_{j=1}^n \mu_m(\alpha_j^l a_j)\leq 2 \mu_m(a_i)$), we may pass to a convergent subsequence and attain the infimum at the limit. The $\ell_2$ norm $\|\cdot\|_2$ is strictly convex, so $\|a_i\|_2<\sum_{j\leq n} \|\alpha_j a_j\|_2$, while $\mu$ by triangle inequality satisfies $\mu(a_i)\leq \sum_{j\leq n} \mu(\alpha_j a_j)$. Since $\mu_m$ is the sum of $\mu$ and a positive multiple of the $\ell_2$ norm, we must have $\mu_m(a_i)<\sum_{j\leq n} \mu_m(\alpha_j a_j)$, a contradiction.
\end{proof}

\begin{notation}
For a finite set $A\subseteq V$ and $P,P'$ partial functions on $V$ (i.e. functions whose domains are subsets of $V$) with $A\subseteq\dom(P),\dom(P')$, we put $d_A(P,P'):=\max_{a\in A} |P(a) - P'(a)|$.
\end{notation}
    
Let $T$ be the countable set of tuples $(n,n',P,P',g)$ such that:
\begin{enumerate}[(a)]
    \item\label{it:first} $n,n'\in \N$;
    \item $P\in \Rat^{\dom(P)}$, $P'\in\Rat^{\dom(P')}$ where $\dom(P)$ and $\dom(P')$ are finite subsets of $V$;
    \item\label{it:existenceOfNorm} there exists $\nu\in\B$ such that $P' = \nu|_{\dom(P')}$;
    \item $g:\dom(P)\to \dom(P')$ is a one-to-one mapping;
    \item $P=P'\circ g$;
    \item\label{it:almostIsometry} whenever $\mu'\in\PP$, $\nu'\in\B$ are such that $P'\subseteq \nu'$, $d_{\dom(P)}(P,\mu')<\tfrac{1}{n}$ and $\mu'$ restricted to $\Span(\dom(P))\subseteq c_{00}$ is a norm then there exists $T_g:(\Span(\dom(P)),\mu')\to (\Span(\dom(P')),\nu')$ which is a $(1+\tfrac{1}{n'})$-isomorphism and $T_g\supseteq g$.
\end{enumerate}
For $(n,n',P,P',g)\in T$, we let $G(n,n',P,P',g)$ be the set of $\mu\in\PP_\infty$ such that
\begin{itemize}
    \item whenever $d_{\dom(P)}(P,\mu)<\tfrac{1}{n}$ and $\mu$ restricted to $\Span (\dom(P))\subseteq c_{00}$ is a norm, there is a $\Rat$-linear mapping $\Phi:V\cap \Span (\dom(P'))\to V$ such that $\mu(\Phi(gx)-x) < \tfrac{2}{n'}\mu(x)$ for every $x\in \dom(P)$ and $|P'(x) - \mu(\Phi(x))|<\tfrac{1}{n'}P'(x)$ for every $x\in \dom(P')$.
\end{itemize}
\begin{proposition}\label{prop:characterization}
Let $\mu\in \PP_\infty$. Then $X_\mu$ is isometric to the Gurari\u{\i} space if and only if $\mu\in G(n,n',P,P',g)$ for every $(n,n',P,P',g)\in T$.
\end{proposition}
\begin{proof}
In order to prove the first implication, let $\mu\in\PP_\infty$ be such that $X_\mu$ is isometric to the Gurari\u{\i} space and let $(n,n',P,P',g)\in T$ be such that $d_{\dom(P)}(P,\mu) < \tfrac{1}{n}$  and $\mu$ restricted to $\Span (\dom(P))\subseteq c_{00}$ is a norm. Consider the finite-dimensional space $A:=(\Span(\dom(P)),\mu)$. Let $\nu\in\B$ be as in \ref{it:existenceOfNorm}. Put $B = (\Span(\dom P'),\nu)$ and pick a basis $\b\subseteq V$ of $B$. 
By \ref{it:almostIsometry} there exists $T_g:A\to B$, which is a $(1+\tfrac{1}{n'})$-isomorphism and $T_g\supseteq g$. By \cite[Lemma 2.2]{KubisSolecki}, there is a $(1+\tfrac{1}{3n'})$-isomorphism $S:B\to X_\mu$ such that $\|ST_g-Id_A\|<\tfrac{1}{n'}$. By Lemma~\ref{lem:approx2}\ref{it:approxIsometryOnBasis}, we may for every $b\in\b$ find $x_b\in V$ such that the linear mapping $Q:S(B)\to X_\mu$ given by $Q(S(b)) = x_b$, $b\in \b$, is a $(1+\tfrac{1}{3n'})$-isomorphism with $\|Q-Id\|<\tfrac{1}{3n'}$. Consider $\Phi = QS|_{V\cap \Span(\dom P')}$. This is indeed a $\Rat$-linear map and since $QS$ is a $(1+\tfrac{1}{3n'})^2$-isomorphism and $(1+\tfrac{1}{3n'})^2<1+\tfrac{1}{n'}$, we have $|\mu(\Phi(x)) - \nu(x)|<\tfrac{1}{n'}\nu(x)$ for $x\in \dom(P')$. Moreover, for every $x\in \dom(P)$ we have 
\[\begin{split}
\mu(\Phi(gx) - x) & = \mu(QST_gx - x)\leq \mu(QST_gx - ST_gx) + \mu(ST_gx - x)\\
& < \tfrac{1}{3n'}\|ST_g\|\mu(x) + \tfrac{1}{n'}\mu(x) \leq \tfrac{2}{n'}\mu(x).
\end{split}\]
This shows that $\mu\in G(n,n',P,P',g)$.

In order to prove the second implication, let $\mu\in \PP_\infty$ be such that $\mu\in G(n,n',P,P',g)$ whenever $(n,n',P,P',g)\in T$. In what follows for $x\in c_{00}$ we denote by $[x]\in X_\mu$ the equivalence class corresponding to $x$. Pick a finite-dimensional space $A\subseteq X_\mu$, and an isometry $G:A\to B$, where $B$ is a finite-dimensional Banach space, we may without loss of generality assume $B \subseteq X_{\mu_B}$ for some $\mu_B\in\B$. Let $\b_A:=\{a_1,\ldots,a_j\}$ be a normalized basis of $A$ and extend $G(\b_A) = \{G(a_1),\ldots,G(a_j)\}$ to a normalized basis $\b_B = \{b_1,\ldots,b_k\}$ of $B$. Fix $\eta > 0$. It suffices to find a $(1+\eta)$-isomorphism $\Psi:B\to X_\mu$ with $\|\Psi G - I_A\|\leq\eta$. Consider the functions $\phi_1$ and $\phi_2^{\b_A}$ from Lemma~\ref{lem:approx} and Lemma~\ref{lem:approx2}\ref{it:approxIsometryOnBasis}. Pick $\delta\in(0,1)$ such that $\max\{\phi_1(t),\phi_2^{\b_A}(t)\}<\eta$ whenever $t<\delta$ and $\varepsilon\in (0,\tfrac{1}{20})$ such that $\phi_1(5\varepsilon)<\tfrac{1}{20}$ and $\varepsilon + 72\max\{\varepsilon,\phi_1(5\varepsilon)\}<\delta$.
\begin{claim}\label{claim:1}There are finite sets $M,N\subseteq V$ such that $\mu$ restricted to $\Span N\subseteq c_{00}$ is a norm and surjective $(1+\varepsilon)$-isomorphisms $T_A:A\to (\Span N,\mu)$, $T_B:B\to (\Span M,\mu_B)$  such that:
\begin{itemize}
	\item $N$ and $M$ are $\varepsilon$-dense sets  for $S_{T_A(A)}$ and $S_{T_B(B)}$, respectively;
	\item we have $\|[T_Aa_i]-a_i\|_{X_\mu}<\varepsilon$ for every $a_i\in\b_A$ and $\|(T_A)^{-1}x-[x]\|_{X_\mu}<\varepsilon$, $|\mu(x)-1|<\varepsilon$ for every $x\in N$;
	\item $(T_B)^{-1}(M)$ is $\tfrac{\varepsilon}{3}$-dense for $S_B$ and $\max\{|\mu_B((T_B)^{-1}x) - 1|, |\mu_B(x) - 1|\}<\tfrac{\varepsilon}{2}$ for every $x\in M$;
	\item $\big(T_BG(T_A)^{-1}\big)(N)\subseteq M$.
\end{itemize}
\end{claim}
\begin{proof}[Proof of Claim~\ref{claim:1}]
	By Lemma~\ref{lem:approx2}\ref{it:approxIsometryOnBasis}, we may pick $\{f_1,\ldots,f_n\}\subseteq V$ such that the linear operator $T_A:A\to X_\mu$ given by $T_A(a_i)=[f_i]$, $i\leq j$, is a $(1+\frac{\varepsilon}{6})$-isomorphism and $\|[T_Ax]-x\|_{X_\mu}<\tfrac{\varepsilon}{6}\|x\|_{X_\mu}$, $x\in A$. This implies that $\mu$ restricted to $\Span \{f_1,\ldots,f_j\}$ is a norm and since $T_A(A)$ is isometric to $(\Span\{f_1,\ldots,f_n\},\mu)$ we consider $T_A$ as a $(1+\frac{\varepsilon}{6})$-isomorphism between $A$ and $(\Span\{f_1,\ldots,f_n\},\mu)$. Now, pick $N'\subseteq A$ a finite $\tfrac{\varepsilon}{6}$-dense set for $S_A$ consisting of rational linear combinations of points from $\b_A$ with $\b_A\subseteq N'$ such that $|\|x\|_{X_\mu}-1|<\tfrac{\varepsilon}{6}$ for every $x\in N'$. Then $\|[T_Ax]-x\|_{X_\mu}<\tfrac{\varepsilon}{6}\|x\|_{X_\mu}<\tfrac{\varepsilon}{6}(1+\tfrac{\varepsilon}{6})<\tfrac{\varepsilon}{5}$ for every $x\in N'$. Put $N:=T_A(N')\subseteq V$. Then we easily obtain $|\mu(x)-1|<\tfrac{\varepsilon}{2}$ for every $x\in N$ and, by Lemma~\ref{lem:approx2}\ref{it:approxNety}, $N$ is $\tfrac{\varepsilon}{2}$-dense in $S_{T_A(A)}$. Similarly as above, we may pick $\{g_1,\ldots,g_k\}\subseteq V$ such that the linear operator $T_B:B\to X_{\mu_B}$ given by $T_B(b_i)=g_i$, $i\leq k$, is a $(1+\frac{\varepsilon}{6})$-isomorphism and we find $M'\subseteq B$ a finite $\tfrac{\varepsilon}{6}$-dense set for $S_B$ consisting of rational linear combinations of points from $\b_B$ with $M'\supseteq \{G(x)\setsep x\in N'\}\cup \b_B$ and $|\mu_B(x) - 1|<\tfrac{\varepsilon}{6}$ for $x\in M'$. Put $M:=T_B(M')$, then similarly as above $|\mu_B(x)-1|<\tfrac{\varepsilon}{2}$ for every $x\in M$ and $M$ is $\tfrac{\varepsilon}{2}$-dense in $S_{T_B(B)}$. Finally, we obviously have $\big(T_BG(T_A)^{-1}\big)(N) = T_B(G(N'))\subseteq T_B(M') = M$.
\end{proof}
By Lemma~\ref{lem:BIsDense}, there is $\nu\in\B$ having rational values on $M$ with $d_{M}(\nu,\mu_B\circ (T_B)^{-1})<\tfrac{\varepsilon}{2}$. Put $P' = \nu|_{M}$, consider the one-to-one map $g:N\to M$ given by $g:=T_BG(T_A)^{-1}|_N$ and put $P = P'\circ g$. Let $n\in\N$ be the integer part of $\tfrac{2}{3\varepsilon}$ and $n'\in\N$ be the integer part of $\frac{1}{9\max\{\varepsilon,\phi_1(5\varepsilon)\}}$.
Easy computations show that $\tfrac{3}{2}\varepsilon\leq \frac{1}{n}<2\varepsilon$ and $9\max\{\varepsilon,\phi_1(5\varepsilon)\}\leq \tfrac{1}{n'} < 18\max\{\varepsilon,\phi_1(5\varepsilon)\}$ (in the last inequality we are using that $\max\{\varepsilon,\phi_1(5\varepsilon)\}<\tfrac{1}{20}$).

Note that for every $x\in M$ we have
\begin{equation}\label{eq:nu}
\max\{|\nu(T_Bx)-1|,|\nu(x)-1|\}\leq \tfrac{\varepsilon}{2} + \max\{|\mu_B(x)-1|,|\mu_B((T_B)^{-1}x)-1|\} < \varepsilon.
\end{equation}
\begin{claim}\label{claim:2}
We have $(n,n',P,P',g)\in T$ and $d_{N}(P,\mu)<\tfrac{1}{n}$.
\end{claim}
\begin{proof}[Proof of Claim~\ref{claim:2}]
In order to see that $d_{N}(P,\mu)<\tfrac{1}{n}$, pick $x\in N$. Then
\[\begin{split}
|P(x) - \mu(x)| \leq \tfrac{\varepsilon}{2} + |\mu_B(G(T_A)^{-1}(x)) - \mu(x)| = \tfrac{\varepsilon}{2} +  |\|(T_A)^{-1}(x)\|_{X_\mu} - \|[x]\|_{X_\mu}| < \tfrac{3}{2}\varepsilon.
\end{split}\]
In order to see that $(n,n',P,P',g)\in T$, let us verify condition \ref{it:almostIsometry}.
Let $\mu'\in \PP$, $\nu'\in\mathcal B$ be such that $P'\subseteq \nu'$, $d_{N}(P,\mu')<\tfrac{1}{n} < 2\varepsilon$ and $\mu'$ restricted to $\Span N\subseteq c_{00}$ is a norm. Note that $|\mu'(x)-1|<5\varepsilon$ for every $x\in N$ and so, since $N$ is $\varepsilon$-dense for the sphere of $T_A(A) = (\Span N,\mu)$, the mapping $id:(\Span N,\mu)\to (\Span N,\mu')$ is a $(1+\phi_1(5\varepsilon))$-isomorphism. Further, $|\nu'(x)-1|=|\nu(x)-1|<\varepsilon$ for every $x\in M$ and so the mapping $id:(\Span M,\mu_B)\to (\Span M,\nu')$ is a $(1+\phi_1(5\varepsilon))$-isomorphism as well. Finally, since $T_BG(T_A)^{-1}$ is a $(1+\varepsilon)^2$-isomorphism between $(\Span N,\mu)$ and $(\Span g(N),\mu_B)$ and 
\[
(1+\phi_1(5\varepsilon))^2(1+\varepsilon)^2 \leq (1+3\phi_1(5\varepsilon))(1+3\varepsilon)\leq 1 + 9\max\{\varepsilon,\phi_1(5\varepsilon)\}\leq 1+\tfrac{1}{n'},
\]
we have that $T_g:=id\circ T_B\circ G\circ (T_A)^{-1}\circ id:(\Span N,\mu')\to (\Span M,\nu')$ is a $(1+\tfrac{1}{n'})$-isomorphism.
\end{proof}

Since $\mu\in G(n,n',P,P',g)$, there is a $\Rat$-linear mapping $\Phi:V\cap (\Span M,\nu)\to V$ such that $\mu(\Phi(gx)-x) < \tfrac{2}{n'}\mu(x)$ for every $x\in N$ and $|\nu(x) - \mu(\Phi(x))|<\tfrac{1}{n'}\nu(x)$ for every $x\in M$. It is easy to see that $\Phi$ extends to a bounded linear operator $\Phi':(\Span M,\nu)\to X_\mu$. Finally, consider $\Psi:=\Phi'\circ T_B:B\to X_\mu$.

For every $x\in M$ we have 
\[|\mu(\Phi(x)) - 1| \leq |\mu(\Phi(x)) - \nu(x)|  + |\nu(x) - 1| \stackrel{\eqref{eq:nu}}{\leq} \tfrac{1}{n'}\nu(x) + \varepsilon \stackrel{\eqref{eq:nu}}{\leq} \tfrac{1}{n'}(1+\varepsilon) + \varepsilon< \delta;\]
thus, $|\|\Psi(x)\|_{X_\mu} - 1| < \delta$ for every $x\in (T_B)^{-1}(M)$ and so $\Psi$ is a $(1+\eta)$-isomorphism.

Further, we have 
\[\begin{split}
\|\Psi G(a_i) - a_i\|_{X_\mu} & \leq \|\Phi(g(T_A a_i)) - [T_A a_i]\|_{X_\mu} + \|[T_A a_i] - a_i\|_{X_\mu}\\
& < \tfrac{2}{n'}(1+\varepsilon)+\varepsilon\le\tfrac{4}{n'}+\varepsilon<\delta;
\end{split}\]
hence, by Lemma~\ref{lem:approx2}\ref{it:approxIsometryOnBasis}, we have $\|\Phi' T_B G - I_A\|\leq \phi_2^{\b_A}(\tfrac{2}{n'}(1+\varepsilon)+\varepsilon) < \eta$.
\end{proof}

\begin{theorem}\label{thm:gurariiGDelta}
Let $\mathbb{G}$ be the Gurari\u{\i} space. Then the isometry class $\isomtrclass[\PP_\infty]{\mathbb{G}}$ is a $G_\delta$-set in $\PP_\infty$.
\end{theorem}
\begin{proof}
By Proposition~\ref{prop:characterization}, we have for the countable set $T$ defined before Proposition~\ref{prop:characterization} that
\[\isomtrclass[\PP_\infty]{\mathbb{G}} = \bigcap_{(n,n',P,P',g)\in T} G(n,n',P,P',g),\]
where $G(n,n',P,P',g)$ is the union of a closed and an open set in $\PP_\infty$ (here we use the observation that the set $\{\mu\in\PP_\infty\colon \mu\text{ restricted to $\Span(\dom P)\subseteq c_{00}$ is a norm}\}$ is open due to Lemma~\ref{lem:infiniteDimIsGDelta}); thus it is the countable intersection of $G_\delta$-sets.
\end{proof}

\begin{proof}[Proof of Theorem~\ref{thm:gurariiTypicalInP}]
Let us recall that $\PP_\infty$ and $\B$ are $G_\delta$-sets in $\PP$, see Corollary \ref{cor:polishTopologies}. Thus, since we have $\isomtrclass[\B]{\mathbb{G}} = \isomtrclass[\PP_\infty]{\mathbb{G}}\cap \B$, it follows from Proposition~\ref{prop:characterization} that $\isomtrclass[\mathcal I]{\mathbb{G}}$ is a $G_\delta$-set in any $\mathcal{I}\in\{\PP,\PP_\infty,\B\}$.

By Corollary~\ref{cor:dense} we also have that $\isomtrclass[\mathcal I]{\mathbb{G}}$ is dense in $\mathcal I$ for every $\mathcal{I}\in\{\PP,\PP_\infty,\B\}$.
\end{proof}

\subsection{Generic objects in $SB(X)$}
In this subsection, we address Problem 5.5 from \cite{GS18} which suggests to investigate generic properties of admissible topologies. We have both positive and negative results. The positive result is Theorem~\ref{thm:GurariiinWijsman} which shows that the isometry class of the Gurari\u{\i} space, as a subset of $SB(\mathbb{G})$, is a dense $G_\delta$-set in the Wijsman topology.
The negative results are Propositions~\ref{prop:Bairecategoryinadmissible} and~\ref{prop:Baireadmiss3}, and Theorem \ref{thm:Baireadmiss2}.
\begin{definition}
Given a closed set $H$ in $X$ we denote by $E^-(H)$ the set $SB(X)\setminus E^+(X\setminus H)$, that is, $E^-(H) = \{F\in SB(X)\setsep F\subseteq H\}$. Obviously, this is a closed set in any admissible topology on $SB(X)$.
\end{definition}

\begin{definition}
Let $X$ be an isometrically universal separable Banach space. By $\tau_W$ we denote the restriction of \emph{the Wijsman topology} from $\F(X)$ to $SB(X)$, that is, the minimal topology on $SB(X)$ such that the mappings $SB(X)\ni F\mapsto \dist_{X}(x,F)$ are continuous for every $x\in X$. Note that $\tau_W$ is admissible, see \cite[Section 2]{GS18}.
\end{definition}

\begin{theorem}\label{thm:GurariiinWijsman}
The isometry class $\isomtrclass{\mathbb{G}}$ is a dense $G_\delta$-set in $(SB(\mathbb{G}),\tau_W)$.
\end{theorem}
\begin{proof}
The isometry class $\isomtrclass{\mathbb{G}}$ is a $G_\delta$-set in $(SB(\mathbb{G}),\tau_W)$ since it is a $G_\delta$-set in $\PP$ (by Theorem~\ref{thm:gurariiTypicalInP})
and there is a continuous reduction from $(SB(\mathbb{G}),\tau_W)$ to $\PP$ by Theorem~\ref{thm:reductionFromSBToP}. So we must show that it is dense.

Choose a basic open set $N$ in $\tau_W$ which is given by some closed subspace $X\subseteq\mathbb{G}$, finitely many points $x_1,\ldots,x_n\in \mathbb{G}$ and $\varepsilon>0$ so that $$N=\{Z\in SB(\mathbb{G})\colon \forall i\leq n\; (| \dist_{\mathbb{G}}(x_i,X)-\dist_{\mathbb{G}}(x_i,Z)|<\varepsilon)\}.$$ Let us find a space $G$ isometric to $\mathbb{G}$ such that $G\in N$. Let $Y$ be $\Span \{X\cup \{x_i\colon i\leq n\}\}$. Since $X$ embeds into both $Y$ and $\mathbb{G}$ we can consider the push-out of that diagram, i.e. the amalgamated sum of $Y$ and $\mathbb{G}$ along the common subspace $X$. Recall this is nothing but the quotient $(\mathbb{G}\oplus_1 Y)/Z$, where $Z=\{(z,-z)\colon z\in X\}$. Denote this space by $G'$ and notice that $\mathbb G$ is naturally embedded into $G'$. It is straightforward to verify that for each $i\leq n$, $\dist_{G'}(x_i,\mathbb{G})=\dist_{\mathbb{G}} (x_i,X)$. Since $\mathbb{G}$ is universal, there is a linear isometric embedding $\iota: G'\hookrightarrow \mathbb{G}$. As there is a linear isometry $\phi:\iota[\Span\{x_i\colon i\leq n\}]\rightarrow \Span\{x_i\colon i\leq n\}$, by \cite[Theorem 1.1]{KubisSolecki} there is a bijective linear isometry $\Phi:\mathbb{G}\rightarrow \mathbb{G}$ such that $\|\Phi\circ \iota(x_i)-x_i\|<\varepsilon$, for each $i\leq n$. By the triangle inequality, it follows that $G:=\Phi\circ\iota[\mathbb{G}\subseteq G']$ satisfies for each $i\leq n$, $|\dist_{\mathbb{G}}(x_i,G)-\dist_{\mathbb{G}}(x_i,X)|<\varepsilon$, so it is the desired space isometric to $\mathbb{G}$ lying in the open set $N$.
\end{proof}

The rest of the section is devoted to negative results. They show that the definition of an admissible topology allows a lot of flexibility by which one can alter which properties should be meager or not.
\begin{proposition}\label{prop:Bairecategoryinadmissible}
Let $X$ be an isometrically universal separable Banach space and let $\tau$ be an admissible topology on $SB(X)$. Then there exists an admissible topology $\tau'\supseteq \tau$  on $SB(X)$ such that the set $\isomrfclass{\Gurarii}$ is nowhere dense in $(SB_\infty(X),\tau')$.
\end{proposition}
\begin{proof}
By the definition of an admissible topology, we may pick $(U_n)_{n\in\N}$, a basis of the topology $\tau$, such that for every $n\in\N$ there are nonempty open sets $V_k^n$, $k=1,\ldots,N_n$, and $W_n$ in $X$ such that the set $U_n'$ defined by
\[
U_n' = \bigcap_{k=1}^{N_n} E^+(V_k^n)\setminus E^+(W_n)
\]
is a nonempty subset of $U_n$.

We \emph{claim} that for every $n\in\N$ there is $F_n\in U_n'$ such that $\Gurarii\not\hookrightarrow F_n$. Indeed, pick an arbitrary $Z\in U_n'$. We may without loss of generality assume there is $H_0\subseteq Z$ with $H_0\simeq \Gurarii$ and since $\Gurarii$ is isometrically universal, there is $H_1\subseteq H_0$ with $H_1\simeq \ell_2$. Now, pick points $v_k\in Z\cap V_k^n$, $k=1,\ldots,N_n$. Then we put $F_n:= \closedSpan{v_1,\ldots,v_{N_n},u\setsep u\in H_1}$. Since $F_n$ is a subset of $Z$, we have $F_n\notin E^+(W_n)$ and since it contains the points $v_1,\ldots,v_{N_n}$, we have $F_n\in U_n'$. Moreover, it is a space isomorphic to $\ell_2$ and so $\Gurarii\not\hookrightarrow F_n$.

Thus, for every $n\in\N$ there is a closed subspace $F_n$ of $X$ such that $U_n\cap E^-(F_n)$ is a nonempty set disjoint from $\isomrfclass{\Gurarii}$.

It is a classical fact, see e.g., \cite[Lemma 13.2 and Lemma 13.3]{KechrisBook}, that the topology $\tau'$ generated by $\tau\cup\{E^-(F_n)\setsep n\in\N\}$ is Polish. It is easy to check it is admissible. Moreover, for every $n\in\N$ we have that $U_n\cap E^-(F_n)$ is a nonempty $\tau'$-open set in $U_n$ disjoint from $\isomrfclass{\Gurarii}$. It follows that nonempty sets of the form $U_n\cap \bigcap_{m\in I}E^-(F_m)$, for finite $I\subseteq \Nat$, give us a $\pi$-basis of $\tau'$. Since obviously each element of the form $U_n\cap \bigcap_{m\in I} E^-(F_m)$ is disjoint from $\isomrfclass{\Gurarii}$, the set $\isomrfclass{\Gurarii}$ is $\tau'$-nowhere dense.

\end{proof}

Actually, one may observe that the same proof gives the following more general result, where the pair $(\mathbb{G},\ell_2)$ is replaced by a more general pair of Banach spaces.

\begin{theorem}\label{thm:Baireadmiss2}
Let $X$ be an isometrically universal separable Banach space and let $\tau$ be an admissible topology on $SB(X)$. Let $Y$ and $Z$ be infinite-dimensional Banach spaces such that $Y\hookrightarrow Z$ and $Z\not\hookrightarrow Y\oplus F$ for every finite-dimensional space $F$.

Then there exists an admissible topology $\tau'\supseteq \tau$ on $SB(X)$ such that the set $\isomrfclass{Z}$ is nowhere dense in $(SB_\infty(X),\tau')$.
\end{theorem}

It is even possible to find an admissible topology $\tau$ such that $\isomtrclass{\ell_2}$ is not a meager set in $(SB_\infty(X),\tau)$ which is an immediate consequence of the following more general observation (the property $(P)$ below would be ``$X$ is isometric to $\ell_2$'').
\begin{proposition}\label{prop:Baireadmiss3}
Let $X$ be an isometrically universal separable Banach space and $\tau$ be an admissible topology on $SB_\infty(X)$. Let $(P)$ be a non-void property (i,e. there are spaces with such a property) of infinite-dimensional Banach spaces closed under taking subspaces. Then there is an admissible topology $\tau'\supseteq \tau$ such that the set $\{Y\in SB_\infty(X)\colon Y\text{ has }(P)\}$ has non-empty interior in $(SB_\infty(X),\tau')$.
\end{proposition}
\begin{proof}
Pick $F\in SB_\infty(X)$ with $(P)$. Using again the classical fact, see e.g. \cite[Lemma 13.2]{KechrisBook}, that the topology $\tau'$ generated by $\tau\cup\{E^-(F)\}$ is Polish, it is easy to check it is admissible. Then the $\tau'$-open set $E^-(F)$ is a subset of $\{Y\in SB_\infty(X)\colon Y\text{ has }(P)\}$.
\end{proof}

\bibliographystyle{siam}
\bibliography{ref}
\end{document}